\newif\ifTwoColumn
\newif\ifTechReport

\TechReporttrue    

\ifTwoColumn
\documentclass[journal]{IEEEtran}
\else
\ifTechReport
\documentclass[journal,12pt,onecolumn,draftcls]{IEEEtran}
\else
\documentclass[journal,11pt,onecolumn,draftclsnofoot]{IEEEtran}
\fi
\fi

\usepackage[utf8]{inputenc} 
\usepackage[T1]{fontenc}    
\usepackage{url}            
\usepackage{booktabs}       
\usepackage{amsfonts,amsmath,amsthm,amssymb,mathrsfs}       
\usepackage{bbm}
\usepackage{cite}
\usepackage{mathtools,xparse}
\usepackage[makeroom]{cancel}
\usepackage[acronym]{glossaries}
\usepackage{glossaries-prefix}
\usepackage{color}
\usepackage{balance}
\usepackage{enumerate,enumitem}
\usepackage{algpseudocode}
\usepackage{subcaption}
\usepackage{nicefrac}       
\usepackage{microtype}      
\usepackage{xcolor}         
\usepackage{graphicx}
\usepackage[prependcaption,colorinlistoftodos]{todonotes}
\usepackage{multirow}
\usepackage{hyperref}
\usepackage{booktabs,tabularx}
\usepackage{cases}
\usepackage[ruled,vlined]{algorithm2e}
\usepackage{xifthen}
\usepackage{eurosym}

\graphicspath{{./figures/}}

\DeclareMathSizes{10}{10}{6}{4}

\newcommand{\norm}[1]{\left\lVert#1\right\rVert}

\newcommand{\continuanceref}{}

\newcommand{\R}{\mathbb{R}}

\newcommand{\N}{\mathbb{N}}

\newcommand{\bbS}{\mathbb{S}}
\newcommand{\mc}[1]{\mathcal{#1}}

\newcommand{\eqdef}{\coloneqq}
\newcommand{\reqdef}{\eqqcolon}

\newcommand{\prob}[2][]{
	\ifthenelse{
		\isempty{#1}}{\mathbb{#2}}{\mathbb{#2}\left\{#1\right\}
	}
}
\newcommand{\expected}[2]{\mathbb{E}_{#1}\left[#2\right]}

\newcommand{\prox}{\mathrm{prox}}
\newcommand{\fix}{\mathrm{fix}}
\newcommand{\zer}{\mathrm{zer}}
\newcommand{\Id}{\mathrm{Id}}

\newcommand{\dom}{\mathrm{dom}}
\newcommand{\proj}{\mathrm{proj}}
\newcommand{\diag}{\mathrm{diag}}

\newcommand{\col}{\mathrm{col}}

\newcommand{\resolvent}[2]{\textrm{#1}_{#2}}

\newcommand{\set}[2]{\left\{ #1 :  #2  \right\}}
\newcommand{\inner}[2]{\langle #1, #2 \rangle}

\newcommand{\bs}[1]{\boldsymbol{#1}}
\newcommand{\bsone}{\boldsymbol{1}}

\newtheorem{theorem}{Theorem}[section]
\newtheorem{definition}[theorem]{Definition}

\newtheorem{lemma}[theorem]{Lemma}

\newtheorem{remark}[theorem]{Remark}
\newtheorem{standing}[theorem]{Standing Assumption}
\newtheorem{assumption}[theorem]{Assumption}

\newcommand{\normaltext}[1]{\textnormal{#1}}

\usepackage[colorinlistoftodos]{todonotes}

\usepackage{xcolor,calc}

\makeglossaries
\newacronym{iid}{i.i.d.}{independent and identically distributed}
\newacronym{wrt}{w.r.t.}{with respect to}
\newacronym{wlog}{w.l.o.g.}{without loss of generality}
\newacronym{PAC}{PAC}{probably approximately correct}
\newacronym{SNEP}{SNEP}{stochastic Nash equilibrium problem}
\newacronym{SNE}{SNE}{stochastic Nash equilibrium}
\newacronym{SGD}{SGD}{stochastic gradient descent}
\newacronym{NE}{NE}{Nash equilibrium}
\newacronym{GNE}{GNE}{generalized Nash equilibrium}
\newacronym{SGNEP}{SGNEP}{stochastic generalized Nash equilibrium problem}
\newacronym{KKT}{KKT}{Karush–Kuhn–Tucker}
\newacronym{FB}{FB}{forward-backward}
\newacronym{FBF}{FBF}{forward-backward-forward}
\newacronym{EG}{EG}{extragradient}
\newglossaryentry{VI}
{
	name={VI},
	description={variational inequality},
	first={\glsentrydesc{VI} (\glsentrytext{VI})},
	plural={VIs},
	descriptionplural={variational inequalities},
	firstplural={\glsentrydescplural{VI} (VIs)}
}
\newacronym{DR}{DR}{Douglas-Rachford}
\newacronym{SAA}{SAA}{sample average approximation}
\newacronym{SA}{SA}{stochastic approximation}
\newacronym{PEV}{PEV}{plug-in electric vehicle}
\newacronym{DSO}{DSO}{distribution system operator}

\begin{document}
	
\title{Finite-sample guarantees for data-driven forward-backward operator methods}
\author{Filippo Fabiani and Barbara Franci
\thanks{F. Fabiani is with the IMT School for Advanced Studies Lucca, Piazza San Francesco 19, 55100, Lucca, Italy ({\tt filippo.fabiani@imtlucca.it}). B. Franci is with the Department of Mathematical Sciences, Politecnico di Torino, Corso Duca degli Abruzzi 24, 10129, Torino, Italy ({\tt barbara.franci@polito.it}).}}


\maketitle

\begin{abstract}
	We establish finite sample certificates on the quality of solutions produced by data-based \gls{FB} operator splitting schemes. As frequently happens in stochastic regimes, we consider the problem of finding a zero of the sum of two operators, where one is either unavailable in closed form or computationally expensive to evaluate, and shall therefore be approximated using a finite number of noisy oracle samples. Under the lens of algorithmic stability, we then derive probabilistic bounds on the distance between a true zero and the \gls{FB} output without making specific assumptions about the underlying data distribution. We show that under weaker conditions ensuring the convergence of \gls{FB} schemes, stability bounds grow proportionally to the number of iterations. Conversely, stronger assumptions yield stability guarantees that are independent of the iteration count.
	We then specialize our results to a popular \gls{FB} stochastic Nash equilibrium seeking algorithm and validate our theoretical bounds on a control problem for smart grids, where the energy price uncertainty is approximated by means of historical data.
\end{abstract}


\IEEEpeerreviewmaketitle

\glsresetall

\section{Introduction}\label{sec:intro}
\IEEEPARstart{F}{inding} a pervasive application in optimization and game theory, operator splitting methods can be generally employed to construct zeros of sums of operators as fixed point iterations \cite{ryu2016primer,bauschke2017correction}. Specifically, they offer an elegant framework and systematic tools for solving problems of the form:
\[
\text{Find }\omega^\star\in\Omega \text{ s.t. } 0\in\mc A(\omega^\star)+\mc B(\omega^\star),
\]
with decision set $\Omega\subseteq\R^n$, and mappings $\mc A:\Omega\to2^\Omega$, $\mc B:\Omega\to2^\Omega$. We then say that $\omega^\star$ is a \emph{zero} of $\mc A+\mc B$, i.e., it belongs to the set $\zer(\mc A+\mc B)\eqdef\set{\omega\in\R^n}{0\in\mc A(\omega)+\mc B(\omega)}$.
In this paper, we will specifically investigate the case in which the operator $\mc A$ is available and its evaluation, or that of the associated resolvent  $\resolvent{J}{\mc A}\eqdef(\Id+\mc A)^{-1}$ (e.g., a proximity operator), is reasonably cheap, while the operator $\mc B$ is either unavailable, or quite challenging to evaluate. The latter shall then be \emph{approximated} in a data-driven fashion by relying on a \emph{finite} number of data.

\subsection{Data-driven approximation methods and related work}
Especially in stochastic regimes, several optimization problems indeed require the evaluation of an expected value mapping, a quantity that can either be hard to compute, or even inaccessible in case the probability distribution of the random variable is unknown. As such, several data-driven approximation schemes have been proposed in the literature \cite{robbins1951,koshal2013,franci2020distributed,franci2021distributed,franci2021stochastic,fabiani2022stochastic,yousefian2017smoothing,shapiro2003monte,shapiro2021lectures,alacaoglu2022stochastic,iusem2017extragradient,gower2020variance}. The latter essentially assume that some batch of random, \gls{iid} samples is available, and then leverage approximation procedures that roughly fall within two large umbrellas: \gls{SAA} and \gls{SA}. In \gls{SAA}-based approaches, one replaces the expected value formulation with the average over a large number of samples of the random variable from a pool of already existing data \cite{shapiro2003monte,shapiro2021lectures}. Successively, the approximated, yet deterministic, problem is solved, and convergence to a solution to the original stochastic problem is proved \emph{when the number of samples grows to infinity}. In contrast, popular \gls{SA} schemes allow one to \emph{sample a realization of the random variable whenever needed}, e.g., at every algorithmic step, thereby originating fully stochastic procedures for computing an optimal solution asymptotically. While \gls{SA}-based algorithms are more computationally attractive compared to those leveraging \gls{SAA}, since they rely on a smaller number of samples at every iteration,
they usually require stronger assumptions on the problem data \cite{koshal2013,yousefian2017smoothing}. 
As a middle ground alternative, algorithms based on variance-reduced SA have been proposed. In particular, there are two available options: consider the average over an increasing number of samples, still drawn at every iteration \cite{franci2020distributed, franci2021stochastic,iusem2017extragradient}, or sample a large (but finite) batch only every so often while one single sample is drawn in the majority of iterations \cite{alacaoglu2022stochastic, gower2020variance}. With a few exceptions \cite{koshal2013, fabiani2022stochastic} however, available procedures prove convergence to an exact solution asymptotically, and hence in case the number of samples necessary for the approximation grows indefinitely. This amounts to a rather strong condition to be met in practice, unless one is able to generate \gls{iid} samples ``for free'' through Monte Carlo-based simulations.

To overcome this issue, several efforts have been made in trying to characterize the performance of iterative schemes when only finite information is available. This line of research goes in the direction of understanding the generalization properties of iterative methods, also referred to as \emph{algorithmic stability} analysis \cite{bousquet2002stability}. Roughly speaking, a stable learner is one for which the learned solution does not change much, \gls{wrt} some loss function employed for the ``training'', with small changes in the sample set. In light of the extensive use of \gls{SGD} for machine learning techniques,
several works analyzed the stability properties of \gls{SGD} in different domains. For instance, \cite{hardt2016train}, which is considered one of the pioneer works on the stability analysis of \gls{SGD} for optimization, obtained stability bounds depending on the Lipschitz and strong convexity constants characterizing the cost function, also investigating the convex and non-convex cases, with resulting bounds that may depend on the number of iterations and learning rate. Similar results were obtained by \cite{bassily2020stability}, extended in \cite{xing2021algorithmic} and follow-up papers for adversarial training. More recently, instead, stability bounds of the \gls{SGD} were discussed also for \glspl{VI} \cite{zhao2024learning} and minmax problems \cite{lei2021stability, ozdaglar2022good}. Finally, a vision not strictly related to \gls{SGD} can be found in \cite{wu2019stability}, which considered distributed learning algorithms for big data. 


\subsection{Summary of contribution}
In this paper we take a practical viewpoint. 
Inspired by a recent trend in system identification \cite{tsiamis2023statistical}, to approximate the operator $\mc B$ we assume to have available a finite number of \gls{iid} samples drawn from an unknown probability distribution with no restrictions. 
Then, by focusing on the \gls{FB} scheme in Algorithm \ref{alg:FB}, which is arguably the most used operator splitting technique yielding first-order iterative methods \cite{ryu2016primer,bauschke2017correction}, we ask ourselves:  how far can we get from some $\omega^\star\in\zer(\mc A+\mc B)$ by running such a scheme with an approximation of $\mc B$ exploiting a finite dataset? 

Let $\omega^{K+1}$ be the output of the resulting data-driven \gls{FB} obtained after $K$ iterations. Since exact convergence to a zero should not be expected in this limited information setup, our goal is hence to characterize $\|\omega^{K+1}-\omega^\star\|$ with rigorous data-based certificates. By leveraging tools proper of the algorithmic stability framework \cite{bousquet2002stability}, we make the following contributions:
\begin{enumerate}
	\item[(i)] We design a tailored loss function in a way that it is representative for our purposes, i.e., to characterize the distance of $\omega^{K+1}$ from some $\omega^\star$;
	\item[(ii)] By considering different monotonicity assumptions on the operators involved, we prove uniform stability of the data-driven \gls{FB} \gls{wrt} such a loss function. In line with what observed in convex optimization \cite{hardt2016train}, we note that:
	\begin{itemize}
		\item[--] The weaker conditions granting convergence of the \gls{FB} scheme lead to a stability bound proportional to $K$; 
		\item[--] Stronger assumptions make stability independent on the number of iterations performed.
	\end{itemize}
	As far as we know, none of the abovementioned works has applied algorithmic stability to operator splitting methods based on data, i.e., to a more general framework than mere convex optimization. Our analysis substantially departs from prior approaches, relying on monotone operator-theoretic tools that have not been addressed so far.
	\item[(iii)] We derive computable expressions for $\varepsilon\ge0$ that depend, among the others, on the amount of data available, so that $\|\omega^{K+1}-\omega^\star\|\le\varepsilon$ holds with arbitrarily high confidence and regardless of the distribution underlying the data;
	\item[(iv)] We apply our bounds to a popular \gls{SNE} seeking algorithm based on the \gls{FB} scheme, thereby establishing certificates on the distance between the output such algorithm produces and an \gls{SNE}.
\end{enumerate}
Our theoretical results, which align with and generalize existing literature on algorithmic stability-based approaches, are finally validated through numerical simulations. Specifically, we analyze a \gls{SNEP} modeling a control problem for smart grids, where energy price uncertainty is approximated by means of historical data.

\subsection{Paper organization and notation}
The rest of the paper is organized as follows: in \S \ref{sec:background} we formalize the data-driven problem addressed and provide preliminary concepts on algorithmic stability. In \S \ref{sec:FB} we discuss uniform stability properties for the \gls{FB} scheme \gls{wrt} a predefined loss function under two different sets of assumptions, as well as derive the resulting distribution-free probabilistic certificates. The latter are then specialized to a popular \gls{FB} stochastic Nash equilibrium seeking algorithm in \S \ref{sec:SGNEP}, while numerical simulations are finally conducted in \S \ref{sec:simulations}.

In the remainder we will use Standing Assumption to postulate properties that hold throughout the paper, while we refer to a specific Assumption only when needed.
\subsubsection*{Standard notation}
$\N$, $\R$ and $\R_{\geq 0}$ denote the set of natural, real, and nonnegative real numbers, respectively. $\N_0 \eqdef \N \cup \{0\}$, while $\bar\R\eqdef\R\cup\{+\infty\}$. $\bbS^{n}$ is the space of $n \times n$ symmetric matrices and $\bbS_{(\succcurlyeq) \succ 0}^{n}$ is the cone of positive (semi-)definite matrices.
The transpose of a matrix $A \in \R^{n \times n}$ is $A^\top$, while
$A \succ 0$ ($\succcurlyeq 0$) denotes its positive (semi)definiteness. For a vector $u \in \R^n$ and a matrix $A \succ 0$, $\|u\|$ denotes the standard Euclidean norm, while $\|\cdot\|_A$ the $A$--induced norm $\|u\|_A \coloneqq \sqrt{u^\top A u}$.
$I_{n}$, $\bsone_n$, and $\bs{0}_n$ denote the $n \times n$ identity matrix, the vector of all $1$, and $0$, respectively (we omit the dimension $n$ whenever clear). 
The operator $\col(\cdot)$ stacks its arguments in column vectors or matrices of compatible dimensions.
For example, given vectors $x_1,\dots,x_N$ with $x_i\in\mathbb{R}^{n_i}$ and $\mc I=\{1,\dots,N \}$, we denote $\bs{x} \eqdef (x_1 ^\top,\dots ,x_N^\top )^\top = \col((x_i)_{i\in\mc I}) \in \R^n$, $n \eqdef \sum_{i\in \mc I} n_i$, and $ \bs{x}_{-i} \eqdef \col(( x_j )_{j\in\mc I\setminus \{i\}})$. With a slight abuse of notation, we sometimes use also $\bs x = (x_i, \bs x_{-i})$.
The uniform distribution on $[a,b]$ is denoted by $\mc U(a,b)$, and the normal distribution with mean $\mu$ and variance $\sigma^2$ by $\mc N(\mu,\sigma^2)$. 
\subsubsection*{Operator-theoretic definitions}
Given a set $\mc X \subseteq \R^n$, $\iota_\mc X:\R^n\to\bar\R$ denotes the associated indicator function, i.e., $\iota_\mc X(u)=0$ if $u\in \mc X$, $\iota_\mc X(u)=+\infty$ otherwise. If $\mc X$ is nonempty and convex, the normal cone of $\mc X$ evaluated at $u$ is the set-valued mapping $\mathrm N_{\mc X} : \R^n\to2^{\R^n}$, defined as $\mathrm N_{\mc X}(u) \coloneqq \set{d \in \R^n}{d^\top(v - u)\leq0, \text{ for all } v \in \mc X}$ if $u\in\mc X$, $\mathrm N_{\mc X}(u)\coloneqq\emptyset$ otherwise. The set of fixed points of an operator $\mc T:\mc X\to\mc X$ is denoted by $\fix(\mc T)\eqdef\set{x\in\mc X}{\mc T(x)=x}$.
Given some function $g:\R^n\to\R$, the proximity and projection mappings are defined as $\prox_{\gamma g}(u)\eqdef\textrm{argmin}_{v\in\R^n}~\{g(v)+\tfrac{1}{2\gamma}\|u-v\|^2\}$, $\gamma>0$, and $\proj_{\mc X}(u)\eqdef\textrm{argmin}_{v\in\mc X}~\tfrac{1}{2}\|u-v\|^2$, respectively.
A mapping $F:\dom( F)\subseteq\R^n\to\R^n$ is $\ell$-Lipschitz continuous if, for some $\ell>0$, $\norm{F(x)-F(y)} \leq \ell\norm{x-y}$ for all $x$, $y \in \dom(F)$; $\eta$-strongly monotone if, for some $\eta>0$, $\inner{F(x)-F(y)}{x-y}\geq \eta\|x-y\|^{2}$ for all $x$, $y \in \dom(F)$; $\beta$-cocoercive if, for all $x$, $y \in \dom(F)$ and for some $\beta>0$, $ \inner{F(x)-F(y)}{x-y} \geq \beta\|F(x)-F(y)\|^{2}$; maximally monotone if there exists no monotone operator $G :C\to\R^n$ such that the graph of $G$ contains that of $F$.

\section{Mathematical background}\label{sec:background}
\begin{algorithm}[t]
	\caption{Forward-backward iterative scheme}\label{alg:FB}
	
	\smallskip
	\textbf{Initialization:} Set $\gamma>0$, $x^0 \in\R^n$
	\smallskip
	
	\textbf{Iteration} $k\in\N_0$\textbf{:}
	\begin{equation}\label{eq:FB}
		\begin{aligned}
			&y^k=x^k-\gamma\mc B(x^k)\\
			&x^{k+1}=\resolvent{J}{\gamma\mc A}(y^k)
		\end{aligned}
	\end{equation}
	\vspace{-.1cm}
\end{algorithm}
As one of the most used operator splitting techniques for designing iterative methods in \gls{NE} seeking \cite{yi2019operator, franci2020distributed} and machine learning \cite{hardt2016train,alacaoglu2022stochastic, gower2020variance}, the \gls{FB} scheme in Algorithm~\ref{alg:FB} is defined with $\omega\eqdef\col(y,x)\in\R^{2n}$, learning rate $\gamma>0$, and operators $\mc A:\Omega\to2^\Omega$, $\mc B:\Omega\to\Omega$, momentarily assumed monotone, $\Omega\subseteq\R^n$. By referring to the scheme in \eqref{eq:FB}, the goal then turns into finding $x^\star\in\zer(\mc A+\mc B)$.
We formalize next the data-driven problem of interest, as well as recall key notions and results on algorithmic stability theory.

\subsection{Problem statement}\label{subsec:problem_statement}
As introduced in \S \ref{sec:intro}, to compute some $x^\star\in\zer(\mc A+\mc B)$ we need to approximate the operator $\mc B$ through a finite number of data. To this end, we will therefore assume to rely on several queries to some \emph{noisy operator oracle}
$\mc O:\Omega\times\Xi\to\Omega$. The latter formally amounts to a Borel function so that, given some $x\in\Omega$ and noise input $\xi\in\Xi\subseteq\R^d$, which is distributed according to an unknown probability $\prob{P}$, provides an unbiased estimate of the operator $\mc B$ as postulated next \cite{zhao2024learning}:
\begin{standing}\label{standing:unbiased}
	For all $x\in\Omega$, $\expected{\prob{P}}{\mc O(x,\xi)}=\mc B(x)$.
	\hfill$\square$
\end{standing}

In our analysis, we will hence leverage some approximation $\hat{\mc B}$ that depends on $s$-data taken from the set $\Xi$. The problem we address here is hence the following: given a finite dataset, how far can we get from a point in $\zer(\mc A+\mc B)$ by employing $\hat{\mc B}$ in \eqref{eq:FB} rather than the true operator $\mc B$? Specifically, we will make use of the following data-based approximation for $\mc B$:
\begin{equation}\label{eq:approxB}
	\hat{\mc B}_s(x)=\hat{\mc B}(x,\mc D_s)\eqdef\frac{1}{s}\sum_{i=1}^{s}\mc O(x,\xi^{(i)}),
\end{equation}
which turns the \gls{FB} iterative scheme \eqref{eq:FB} into the set of instructions reported in Algorithm~\ref{alg:FBmod}.
We will then refer to \eqref{eq:FBmod} as the data-driven variant of the \gls{FB} in \eqref{eq:FB}. In the remainder, we will thus implicitly assume that:
\begin{itemize}
	\item[(i)] A set $\mc D_s\eqdef\{\xi^{(i)}\}_{i=1}^s\in\Xi^s$, consisting of $s$ \gls{iid} samples drawn from an \emph{unknown} probability measure $\prob{P}$ attached to $\Xi^s$ is available;
	\item[(ii)] Akin to \cite{hardt2016train,farnia2021train,zhao2024learning}, the oracle $\mc O$ possesses the same properties postulated for the true operator $\mc B$;
	\item[(iii)] The learning rate $\gamma$ is given to meet the requirements discussed later when introducing suitable assumptions. 
\end{itemize}

\begin{standing}\label{standing:boundedness}
	There exists $M>0$ such that, for all $x\in\R^n$, $\norm{\mc B(x)}\le M$.
	\hfill$\square$
\end{standing}

\begin{algorithm}[h]
	\caption{Data-driven \gls{FB} iterative scheme}\label{alg:FBmod}
	
	\smallskip
	\textbf{Initialization:} Samples $\mc D_s$, set $\gamma>0$, $x^0 \in\R^n$
	\smallskip
	
	\textbf{Iteration} $k\in\N_0$\textbf{:}
	\begin{equation}\label{eq:FBmod}
		\begin{aligned}
			&y^k=x^k-\gamma\hat{\mc B}_s(x^k)=x^k-\frac{\gamma}{s}\sum_{i=1}^{s}\mc O(x^k,\xi^{(i)})\\
			&x^{k+1}=\resolvent{J}{\gamma\mc A}(y^k)
		\end{aligned}
	\end{equation}
	\vspace{-.1cm}
\end{algorithm}

While relying on \gls{iid} samples may represent the main practical limitation for applying the probabilistic bounds we will develop, especially for mere control applications, we remark that, once the dataset $\mc D_s\in\Xi^s$ is given, Algorithm~\ref{alg:FBmod} happens to be \emph{deterministic}. Specifically, after running the underlying scheme for $K\ge1$ iterations, we obtain a deterministic output since all the calculations performed do not involve any source of randomness. In addition, Algorithm~\ref{alg:FBmod} turns out to be naturally \emph{symmetric} \gls{wrt} any $\mc D_s$, since the output obtained after $K$ steps does not depend on the order of the elements in $\mc D_s$.

Our goal is then to drive some empirical gap function (defined later) to zero by training on the operator $\hat{\mc B}$ for several iterations, and only seek to control the generalization gap that readily gives a measure on how far we can get from a point in $\zer(\mc A+\mc B)$. In particular, we will be interested in determining the radius $\varepsilon\ge0$ characterizing the set of $\varepsilon$-zeros, defined as: 
\begin{equation}\label{eq:eps_zero}
	\zer_\varepsilon(\mc A+\mc B)\eqdef\set{\omega\in\Omega}{\exists z\in\mc A(\omega) \textrm{ s.t. } \|z+\mc B(\omega)\|\le\varepsilon}.
\end{equation}
Our task will be accomplished by taking an algorithmic stability perspective \cite{bousquet2002stability}, as introduced next.

\subsection{Preliminaries on algorithmic stability}\label{subsec:alg_stab}
An \emph{algorithm} is formally defined as an indexed family of mappings $\{A_s\}_{s\ge 0}$, with $A_s : \Xi^s \to \R^{2n}$ taking some dataset with $s$ samples and returning a deterministic hypothesis $H_s \eqdef A_s(\{\xi^{(i)}\}_{i=1}^s)=A_s(\mc D_s)$ \cite{vidyasagar2013learning}. We will consider later the following sets associated with $\mc D_s$:
\begin{itemize}
	\item[--] The set obtained by \emph{removing} the $i$-th element from $\mc D_s$:
	\[
		\mc D_{s-1}^{\setminus i} = \{\xi^{(1)}, \ldots, \xi^{(i-1)}, \xi^{(i+1)}, \ldots, \xi^{(s)}\}.
	\]
	In short, we also denote $H_{s^{\setminus i}} \eqdef A_{s-1}(\mc D_{s-1}^{\setminus i})$;
	\item[--] The set obtained by \emph{replacing} the $i$-th element from $\mc D_s$:
	\[
		\mc D_s^i = \{\xi^{(1)}, \ldots, \xi^{(i-1)}, \xi', \xi^{(i+1)}, \ldots, \xi^{(s)}\},
	\]
	where $\xi'\in\Xi$ is drawn according to $\prob{P}$, \gls{iid} \gls{wrt} $\mc D_s \setminus \{\xi^{(i)}\}$. Here, we indicate $H_{s^{i}} \eqdef A_s(\mc D_s^i)$.
\end{itemize}

The performance of an algorithm is generally evaluated through some function $\ell : \R^{2n}  \times \R^d \to \R_{\ge 0}$ that measures the
\emph{loss} associated to an hypothesis $H$ \gls{wrt} an example $\xi$.
Attached to the loss function, one identifies the so-called \emph{risk} or \emph{generalization error}, which coincides with a random variable (the hypothesis generated indeed depends on $\mc D_s$), defined as:
\begin{equation}\label{eq:risk}
	r(A,s)=\expected{\prob{P}}{\ell(H_s, \xi)},
\end{equation}
where we use $A$ instead of $A_s$ as first argument. Note that computing \eqref{eq:risk} would require $\prob{P}$, which is however unavailable in the considered framework. Nevertheless, the simplest estimator for \eqref{eq:risk} amounts to the \emph{empirical error}, which reads as:
\begin{equation}\label{eq:empirical_risk}
	\hat r(A,s)=\frac{1}{s} \sum_{i=1}^{s} \ell(H_s, \xi^{(i)}).
\end{equation}

\begin{definition}[\textup{\hspace{-.02cm}\cite[Def.~6]{bousquet2002stability} Uniform stability}]\label{def:unif_stab}
	An algorithm $\{A_s\}_{s\ge0}$ has uniform stability $\beta=\beta(s)$ \gls{wrt} the loss function $\ell$ if for all $s\ge0$, dataset with $s$ samples $\mc D_s$, and $i\in\{1,\ldots,s\}$, $|\ell(H_s,\xi)-\ell(H_{s^{\setminus i}},\xi)| \le \beta$ for all $\xi\in\Xi$.
	\hfill$\square$
\end{definition}

A mapping is then uniformly stable if changing one sample $\xi^{(i)}$ in its input dataset does not significantly alter the output \gls{wrt} to the metric identified by the loss $\ell$ itself. 
Definition~\ref{def:unif_stab} is also frequently stated as $|\ell(H_s,\xi)-\ell(H_{s^{i}},\xi)| \le \beta$ for all $\xi\in\Xi$ \cite{kutin2002almost,bousquet2020sharper}, i.e., considering sample replacement rather than removal. This is indeed how it will be employed later.

The following result will be key to establish our sample complexity bounds on the generalization error, which we will see that, in some cases, it provides a direct upper bound on the distance from some $x^\star\in\zer(\mc A+\mc B)$. An appropriate choice for $\ell$ will hence be inevitably crucial for our purposes:
\begin{lemma}[\textup{\hspace{-.02cm}\cite[Th.~12]{bousquet2002stability} Exponential bound with uniform stability}]\label{lemma:exp_bound}
	Let $\{A_s\}_{s\ge0}$ be an algorithm with uniform stability $\beta=\beta(s)$ \gls{wrt} a loss function $0\le\ell(H_s,\xi)\le\bar \ell$, $\bar \ell\ge0$, for all $\xi\in\Xi$ and $\mc D_s$. Then, for any $s\ge1$ and $\delta\in(0,1)$, the following bound hold true:
	\ifTwoColumn
\else
		\[
		\prob{P}^{s}\left\{\mc D_s\in\Xi^s:r(A,s)\le\hat r(A,s)+2\beta+(4s\beta+\bar \ell)\sqrt{\frac{\ln(1/\delta)}{2s}}\right\}\ge1-\delta.
		\]
	\fi
	\hfill$\square$
\end{lemma}

Based on the McDiarmid's inequality \cite{mcdiarmid1989method}, Lemma~\ref{lemma:exp_bound} says that, with arbitrary confidence $1-\delta$, the risk associated to hypothesis $H_s=A_s(\mc D_s)$ is upper bounded by quantities that: i) if $\beta$ scales inversely with $s$, vanishes as $s\to\infty$ (consistency property), and ii) it is exponential in the confidence parameter $\delta$. Referring to i), we will prove later that it is indeed the case for the data-driven version of the \gls{FB} scheme in Algorithm~\ref{alg:FBmod}, once identified $\omega=\col(y,x)$ as the hypothesis of our algorithm. Since we will consider sample replacement rather than removal, after revisiting the proof of Lemma~\ref{lemma:exp_bound} we will employ the following slightly different bound \cite[Th.~3.2]{kutin2002almost}:
\ifTwoColumn
\else
	\begin{equation}\label{eq:exp_bound_replacement}
		\prob{P}^{s}\left\{\mc D_s\in\Xi^s:r(A,s)\le\hat r(A,s)+\beta+(s\beta+\bar \ell)\sqrt{\frac{2\ln(1/\delta)}{s}}\right\}\ge1-\delta.
	\end{equation}
\fi
We finally note that, since the RHS in both bounds depends on $\hat r(A,s)$, it is unpractical to design loss functions that depend on possibly unknown quantities such as, e.g., some $x^\star\in\zer(\mc A+\mc B)$ itself. This is simply because the term $\sum_{i=1}^{s} \ell(H_s, \xi^{(i)})$ in the RHS might not be computed directly.

\section{Generalization bounds for the\\ data-driven \gls{FB} algorithm}\label{sec:FB}
We now establish stability properties for Algorithm~\ref{alg:FBmod} \gls{wrt} a predefined loss function under two different sets of assumptions commonly used to show convergence of Algorithm~\ref{alg:FB}. 

Then, by letting $H=\omega$ as the generic hypothesis of our data-based \gls{FB} in \eqref{eq:FBloss}, we consider the following loss function as a measure of the associated performance, for a given $\gamma>0$:
\begin{equation}\label{eq:FBloss}
	\begin{aligned}
		\ell(H,\xi)&=\norm{[0_n\quad I_n] H-\gamma \mc O([0_n\quad I_n] H,\xi)-[I_n\quad0_n]H}\\
		&=\norm{x-\gamma \mc O(x,\xi)-y}.
	\end{aligned}
\end{equation}
According to the requirements of Lemma~\ref{lemma:exp_bound}, we need to impose the following mild condition on the underlying loss:
\begin{standing}
	For all $\xi\in\Xi$ and $\mc D_s\in\Xi^s$, $0\le\ell(H_s,\xi)\le\bar \ell$, for some $\bar \ell\ge0$.
	\hfill$\square$
\end{standing}

Although not the only possible choice, it will be clear from Theorems~\ref{th:FB_sample_complexity} and \ref{th:FB_sample_complexity_general} and related proofs why such a loss function is particularly convenient for our purposes. As we will note indeed, besides capturing the fixed-point residual, such a specific choice will make our bounds particularly easy to compute in practical cases---see the discussion in \S \ref{subsec:discussion}.

\subsection{Stability bounds independent on the number of iterations}\label{sec_stab_strong}
Before proceeding in the analysis, let us introduce some properties of the involved operators.
\begin{assumption}\label{ass:SMON}
	The operator $\mc A:\Omega\to2^\Omega$ is maximally monotone, while $\mc B:\Omega\to\Omega$ is $\mu$-strongly monotone and $\kappa$-Lipschitz continuous, with $\mu\le\kappa$, and $\gamma\in(0,2\mu/\kappa^2)$.
	\hfill$\square$
\end{assumption}
Under Assumption~\ref{ass:SMON}, the sequence $\{x^k\}_{k\in\N}$ produced by Algorithm~\ref{alg:FB} converges linearly to the unique $x^\star\in\zer(\mc A+\mc B)$ \cite[Prop.~25.9]{bauschke2017correction}. According to the discussion in \S \ref{subsec:problem_statement}, throughout this subsection we will assume that, for all $\xi\in\Xi$, $x\mapsto\mc O(x,\xi)$ is $\mu$-strongly monotone and $\kappa$-Lipschitz continuous.
We can then preliminary establish what follows:

\begin{lemma}\label{lemma:FB_unif_stab}
	Let $\tau\eqdef\sqrt{1-\gamma(2\mu-\gamma\kappa^2)}<1$. Under Assumption~\ref{ass:SMON}, Algorithm~\ref{alg:FBmod} possesses $(2\gamma M(1+\tau)/(s(1-\tau)))$-uniform stability \gls{wrt} the loss function $\ell(H,\xi)$ in \eqref{eq:FBloss}.
	\hfill$\square$
\end{lemma}
\begin{proof}
	Fix some $s\geq1$, and consider two datasets, $\mc D_s$, $\mc D_s^i\in\Xi^s$, both consisting of $s$-\gls{iid} random samples and differing on the $i$-th instance only, i.e., some $\xi'$ replaces $\xi^{(i)}$ in $\mc D_s^i$. After iterating Algorithm~\ref{alg:FBmod} for $K\ge1$ times by leveraging samples in $\mc D_s$ and $\mc D_s^i$, and starting from the same $x^0\in\R^n$, we obtain $H_s=\omega^K$ and $H_{s^i}=\tilde{\omega}^K$ as the two corresponding hypotheses. By picking an arbitrary $\xi\in\Xi$ we readily have: 
	\ifTwoColumn
\else
		\begin{equation}\label{eq_lemma2}
			\begin{aligned}
				|\ell(H_s,\xi)-\ell(H_{s^{i}},\xi)|&\stackrel{(a)}{\le}\norm{
					x^{K+1}-\gamma \mc O(x^{K+1},\xi)-\tilde x^{K+1}+\gamma \mc O(\tilde x^{K+1},\xi)
					-y^K+\tilde y^K}\\
				&\stackrel{(b)}{\le}\tau\norm{x^{K+1}-\tilde x^{K+1}}+\norm{y^K-\tilde y^K}\\
				&=\tau\norm{\resolvent{J}{\gamma\mc A}(y^K)-\resolvent{J}{\gamma\mc A}(\tilde y^K)}+\norm{y^K-\tilde y^K}\\
				&\stackrel{(c)}{\le}(1+\tau)\norm{y^K-\tilde y^K},
			\end{aligned}
		\end{equation}
	\fi
	\sloppy where (a) follows in view of the reverse triangle inequality, (b) since standard calculations exploiting the $\mu$-strong monotonicity and $\kappa$-Lipschitz continuity of the operator $\mc O(\cdot,\xi)$ reveals that: 
	\[
	\begin{aligned}
		&\norm{(x^{K+1}-\gamma \mc O(x^{K+1},\xi))-(\tilde x^{K+1}-\gamma \mc O(\tilde x^{K+1},\xi))}^2\\
		&=\!\norm{x^{K\!+\!1}\!-\!\tilde x^{K\!+\!1}}^2\!-\!2\gamma\inner{x^{K\!+\!1}\!-\!\tilde x^{K\!+\!1}}{\mc O(x^{K\!+\!1},\xi)\!-\! \mc O(\tilde x^{K\!+\!1},\xi)}\\
		&\hspace{4cm}+\gamma^2\norm{\mc O(x^{K+1},\xi)-\mc O(\tilde x^{K+1},\xi)}^2\\
		&\le\norm{x^{K\!+\!1}\!-\!\tilde x^{K\!+\!1}}^2\!-\!2\mu\gamma\norm{x^{K\!+\!1}\!-\!\tilde x^{K\!+\!1}}^2\!+\!\gamma^2\kappa^2\norm{x^{K\!+\!1}\!-\!\tilde x^{K\!+\!1}}^2\\
		&=(1-2\mu\gamma+\gamma^2\kappa^2)\norm{x^{K+1}-\tilde x^{K+1}}^2\reqdef\tau^2\norm{x^{K+1}-\tilde x^{K+1}}^2,
	\end{aligned}
	\]
	with $\tau\in[0,1)$ since $\gamma\in(0,2\mu/\kappa^2)$ and $\mu\le\kappa$ in view of Assumption~\ref{ass:SMON}, and (c) since $\resolvent{J}{\gamma\mc A}(\cdot)$ is the resolvent of a maximally monotone operator $\mc A$, and hence firmly nonexpansive \cite[Cor.~23.10]{bauschke2017correction}, i.e., for all $\gamma>0$ it satisfies $\norm{\resolvent{J}{\gamma\mc A}(y^K)-\resolvent{J}{\gamma\mc A}(\tilde y^K)}^2\le\norm{y^K-\tilde y^K}^2-\norm{(y^K-\resolvent{J}{\gamma\mc A}(y^K))-(\tilde y^K-\resolvent{J}{\gamma\mc A}(\tilde y^K))}^2\le\norm{y^K-\tilde y^K}^2$.
	According to Definition~\ref{def:unif_stab}, it then remains to bound $\|y^K-\tilde y^K\|$, possibly as a function of $1/s$.
	To this end, we will exploit the generalized growth lemma established in \cite[Lemma~2]{farnia2021train}. Note that $\|y^K-\tilde y^K\|=\|x^K-\frac{\gamma}{s} \sum_{j=1, j\neq i}^s  \mc O(x^K,\xi^{(j)})-\frac{\gamma}{s} \mc O(x^K,\xi^{(i)})-(\tilde x^K-\frac{\gamma}{s} \sum_{j=1, j\neq i}^s \mc O(\tilde x^K,\xi^{(j)})-\frac{\gamma}{s} \mc O(\tilde x^K,\xi'))\|$. Then, for each of the two trajectories yielding $y^K$ and $\tilde y^K$, let us define the following update processes:
	\[
	\begin{aligned}
		&P(x)\eqdef x-\frac{\gamma}{s}\sum_{\stackrel{j=1}{j\neq i}}^s \mc O(x,\xi^{(j)}), &&& \hat P_k(x)\eqdef-\frac{\gamma}{s} \mc O(x,\xi^{(i)})\\
		&P'(x)\eqdef P(x), &&&\hat P'_k(x)\eqdef-\frac{\gamma}{s} \mc O(x,\xi').
	\end{aligned}
	\]
	\sloppy In this way, at the generic iteration $k$ we have that $\|y^k-\tilde y^k\|=\|x^k-\frac{\gamma}{s} \sum_{j=1,j\neq i}^s  \mc O(x^k,\xi^{(j)})-\frac{\gamma}{s} \mc O(x^k,\xi^{(i)})-\tilde x^k+\frac{\gamma}{s} \sum_{j=1,j\neq i}^s  \mc O(\tilde x^k,\xi^{(j)})+\frac{\gamma}{s} \mc O(\tilde x^k,\xi')\|=\|P(x^k)+\hat P_k(x^k)-P'(\tilde x^k)-\hat P'_k(\tilde x^k)\|$. Thus, by relying on \cite[Lemma~2]{farnia2021train}, we obtain that $\|y^k-\tilde y^k\|\le\theta\|x^k-\tilde x^k\|+\|\hat P_k(x^k)\|+\|\hat P'_k(\tilde x^k)\|$, where $\theta$ is the parameter of contraction of the update rule $P(x)+\hat P_k(x)$, which in view of the calculations above is upper bounded by $\tau$, smaller than $1$ if $\gamma\in(0,2\mu/\kappa^2)$ and $\mu\le\kappa$. In view of the boundedness of the operator $\mc O(\cdot,\xi)$, obtained from Standing Assumption~\ref{standing:boundedness}, both $\|\hat P_k(x^k)\|$ and $\|\hat P'_k(\tilde x^k)\|$ are then upper bounded by $\gamma M/s$, thereby yielding:
	\ifTwoColumn
\else
		\[
		\begin{aligned}
			\|y^k-\tilde y^k\|&\le\tau\|x^k-\tilde x^k\|+\frac{2\gamma M}{s}=\tau\|\resolvent{J}{\gamma\mc A}(y^{k-1})-\resolvent{J}{\gamma\mc A}(\tilde y^{k-1})\|+\frac{2\gamma M}{s}\\
			&\le\tau\norm{y^{k-1}-\tilde y^{k-1}}+\frac{2\gamma M}{s}.
		\end{aligned}
		\]
	\fi
	Recursing over the $K$-steps by considering that i) the resolvent operator $\resolvent{J}{\gamma\mc A}(\cdot)$ is firmly nonexpansive, and hence $\|\resolvent{J}{\gamma\mc A}(y^{k-1})-\resolvent{J}{\gamma\mc A}(\tilde y^{k-1})\|\le\norm{y^{k-1}-\tilde y^{k-1}}$ at every $k$, ii) both trajectories originate from the same initial condition $x^0\in\R^n$, and iii) exploiting the relation $\sum_{j=1}^{K} a^{K-j}=(1-a^K)/(1-a)$, which holds true for any $a\in\R$, we obtain:
	\[
	\|y^K-\tilde y^K\|\le\frac{2\gamma M}{s}\sum_{j=1}^K \tau^{K-j}=\frac{2\gamma M(1-\tau^K)}{s(1-\tau)}\le\frac{2\gamma M}{s(1-\tau)},
	\]
	Putting everything together with \eqref{eq_lemma2}, we finally arrive at the following relation:
	\[
	|\ell(H_s,\xi)-\ell(H_{s^{i}},\xi)|\le\frac{2\gamma M(1+\tau)}{s(1-\tau)},
	\]
	which proves uniform stability with bound inverse to $s$.
\end{proof}

We can then prove one of the main results of this section:
\begin{theorem}\label{th:FB_sample_complexity}
	Fix $s\ge1$ and $\delta\in(0,1)$. Given any dataset $\mc D_s\in\Xi^s$, under Assumption~\ref{ass:SMON} there exists $\varepsilon=\varepsilon(s,\delta)>0$ such that $x^{K+1}\in\zer_\varepsilon(\mc A+\mc B)$ with probability at least $1-\delta$, where $x^{K+1}$ is obtained by iterating Algorithm~\ref{alg:FBmod} $K$-times.
	\hfill$\square$
\end{theorem}
\begin{proof}
	From \cite[Prop.~25.1(iv)]{bauschke2017correction}, we have that $\zer(\mc A+\mc B)=\fix(\resolvent{J}{\gamma\mc A}\circ(\Id-\gamma\mc B))$. Therefore, a possible metric to evaluate the distance between a generic $x^{k+1}$ and $x^\star$ is $\|\resolvent{J}{\gamma\mc A}(x^{k+1}-\gamma\mc B(x^{k+1}))-x^{k+1}\|$, which in turn measures the gap between consecutive iterations on variable $x$. After running Algorithm~\ref{alg:FBmod} $K$-times, consider the following relations:
	\ifTwoColumn
\else
		\[
		\begin{aligned}
			\norm{\resolvent{J}{\gamma\mc A}(x^\star-\gamma\mc B(x^\star))-x^\star}&=0\\
			&\le\norm{\resolvent{J}{\gamma\mc A}(x^{K+1}-\gamma\mc B(x^{K+1}))-x^{K+1}}\\
			&\stackrel{(a)}{=}\norm{\resolvent{J}{\gamma\mc A}(x^{K+1}-\gamma\expected{\prob{P}}{\mc O(x^{K+1},\xi)})-\resolvent{J}{\gamma\mc A}(y^K)}\\
			&\stackrel{(b)}{\le}\norm{x^{K+1}-\gamma\expected{\prob{P}}{\mc O(x^{K+1},\xi)}-y^K}\\
			&\stackrel{(c)}{=}\norm{\expected{\prob{P}}{x^{K+1}-\gamma \mc O(x^{K+1},\xi)-y^K}}\\
			&\stackrel{(d)}{\le}\expected{\prob{P}}{\norm{x^{K+1}-\gamma \mc O(x^{K+1},\xi)-y^K}}\\
			&=\expected{\prob{P}}{\ell(H_s,\xi)}=r(A,s),
		\end{aligned}
		\]
	\fi
	where we have used in (a) the fact that the oracle operator is unbiased in view of Standing Assumption~\ref{standing:unbiased}, together with the second relation in \eqref{eq:FBmod}, in (b) the firm nonexpansiveness of the resolvent of $\mc A$, for any $\gamma>0$, in (c) the fact that the data-driven \gls{FB} in \eqref{eq:FBmod} yields a deterministic output $\omega^K$ once fixed the dataset $\mc D_s$, as well as the linearity of the expected value, and in (d) the Jensen's inequality \cite{perlman1974jensen}.
	Thus, upper bounding the risk $\expected{\prob{P}}{\ell(H_s,\xi)}$ with loss function as in \eqref{eq:FBloss} provides an equivalent information on the distance of $x^{K+1}$ from a fixed point of the dynamics produced by iterating the operator $\resolvent{J}{\gamma\mc A}\circ(\Id-\gamma\mc B)$ with perfect knowledge on $\mc B$.
	Therefore, with probability at least $1-\delta$, applying the bound in \eqref{eq:exp_bound_replacement} leads to $\norm{\resolvent{J}{\gamma\mc A}(x^{K+1}-\gamma\mc B(x^{K+1}))-x^{K+1}}\le\epsilon$ with
	\ifTwoColumn
\else
		\[
			\epsilon\eqdef\frac{1}{s} \sum_{i=1}^{s}\norm{x^{K+1}-\gamma \mc O(x^{K+1},\xi^{(i)})-y^K}+\frac{2\gamma M(1+\tau)}{s(1-\tau)}+\left(\frac{2\gamma M(1+\tau)}{1-\tau}+\bar\ell\right)\sqrt{\frac{2\ln(1/\delta)}{s}}.
		\]
	\fi
	This allows us to conclude that $x^{K+1}$ is an $\varepsilon$-zero of the sum of $\mc A$ and $\mc B$, according to \eqref{eq:eps_zero}, with $\varepsilon\eqdef\epsilon/\gamma$. In fact, since $\norm{\resolvent{J}{\gamma\mc A}(x^{K+1}-\gamma\mc B(x^{K+1}))-x^{K+1}}\le\epsilon$, we directly have:
		\[
		\begin{aligned}
			\epsilon&\ge
			\norm{(\Id+\gamma\mc A)^{-1}\circ(\Id-\gamma\mc B)(x^{K+1})-x^{K+1}}\\
			&=\norm{(\Id-\gamma\mc B)(x^{K+1})-(\Id+\gamma\mc A)(x^{K+1})}\\
			&=\norm{-\gamma\mc B(x^{K+1})-\gamma z}=\gamma\norm{z+\mc B(x^{K+1})},
		\end{aligned}
		\]
		in view of the single-valuedness of the resolvent $\resolvent{J}{\gamma\mc A}(\cdot)$ and its properties, where $z\in\mc A(x^{K+1})$, concluding the proof.
\end{proof}

\subsection{Weaker assumptions yield $K$-dependent stability bounds}\label{sec:coco}
Compared to \S \ref{sec_stab_strong}, we now weaken the assumptions on the operator $\mc B$ and, as a consequence, on the oracle $\mc O$.
\begin{assumption}\label{ass:COCO}
	The operator $\mc A:\Omega\to2^\Omega$ is maximally monotone, while $\mc B:\Omega\to\Omega$ is $\theta$-cocoercive, $\theta>0$, and $\gamma\in(0,2\theta)$.
	\hfill$\square$
\end{assumption}
According to \cite[Th.~25.8]{bauschke2017correction}, under Assumption~\ref{ass:COCO} the sequence $\{x^k\}_{k\in\N}$ produced by Algorithm~\ref{alg:FB} converges to some $x^\star\in\zer(\mc A+\mc B)$ (following the notation of \cite[Th.~25.8]{bauschke2017correction}, $\lambda^k=1$ is indeed a valid choice since $\delta$ there is strictly greater than $1$).
In view of the discussion in \S \ref{subsec:problem_statement}, the noisy oracle inherits the same conditions made on $\mc B$. In this case, $\mc O(\cdot,\xi)$ is hence $\theta$-cocoercive for a given $\xi\in\Xi$. We have the following result:
\begin{lemma}\label{lemma:cocoercivity}
	For fixed $s>0$, let $\{\xi^{(1)},\ldots,\xi^{(s)}\}$ be a given set of samples with $\xi^{(i)}\in\Xi$, for all $i=1,\ldots,s$. Then, the operator $\sum_{i=1}^s \mc O(\cdot,\xi^{(i)})$ is $(\theta/s)$-cocoercive.
	\hfill$\square$
\end{lemma}
\begin{proof}
	Given any $x$, $y\in\R^n$, the following relations hold true:
	\ifTwoColumn
\else
		\[
		\begin{aligned}
			\norm{\sum_{i=1}^s \mc O(x,\xi^{(i)})-\sum_{i=1}^s \mc O(y,\xi^{(i)})}^2&\le s \sum_{i=1}^s \norm{\mc O(x,\xi^{(i)})-\mc O(y,\xi^{(i)})}^2\\
			&\le \frac{s}{\theta} \sum_{i=1}^s\inner{x-y}{\mc O(x,\xi^{(i)})-\mc O(y,\xi^{(i)})}\\
			&=\frac{s}{\theta}\inner{x-y}{\sum_{i=1}^s (\mc O(x,\xi^{(i)})-\mc O(y,\xi^{(i)}))}\\
			&=\frac{s}{\theta}\inner{x-y}{\sum_{i=1}^s \mc O(x,\xi^{(i)})- \sum_{i=1}^s \mc O(y,\xi^{(i)})},
		\end{aligned}
		\]	
	\fi
	where in the second inequality we have used the $\theta$-cocoercivity of $\mc O(\cdot,\xi)$ for fixed $\xi\in\Xi$.
\end{proof}
We then have uniform stability of Algorithm~\ref{alg:FBmod} depending on the number of iterations: 
\begin{lemma}\label{lemma:FB_unif_stab_general}
	Under Assumption~\ref{ass:COCO}, Algorithm~\ref{alg:FBmod} possesses $(4\gamma MK/s)$-uniform stability \gls{wrt} the loss function $\ell(H,\xi)$ in~\eqref{eq:FBloss}.
	\hfill$\square$
\end{lemma}
\begin{proof}
	Fix some $s\geq1$, and consider two datasets, $\mc D_s$, $\mc D_s^i\in\Xi^s$, both consisting of $s$-\gls{iid} random samples and differing on the $i$-th instance only, i.e., some $\xi'$ replaces $\xi^{(i)}$ in $\mc D_s^i$. After iterating Algorithm~\ref{alg:FBmod} for $K\ge1$ times by leveraging samples in $\mc D_s$ and $\mc D_s^i$, and starting from the same $x^0\in\R^n$, we obtain $H_s=\omega^K$ and $H_{s^i}=\tilde{\omega}^K$ as the two corresponding hypotheses. By picking an arbitrary $\xi\in\Xi$ we immediately obtain: 
	\ifTwoColumn
\else 
		\[
		\begin{aligned}
			|\ell(H_s,\xi)-\ell(H_{s^{i}},\xi)|&\stackrel{(a)}{\le}\norm{
				x^{K+1}-\gamma \mc O(x^{K+1},\xi)-\tilde x^{K+1}+\gamma \mc O(\tilde x^{K+1},\xi)
				-y^K+\tilde y^K}\\
			&\stackrel{(b)}{\le}\norm{x^{K+1}-\tilde x^{K+1}}+\norm{y^K-\tilde y^K}\\
			&=\norm{\resolvent{J}{\gamma\mc A}(y^K)-\resolvent{J}{\gamma\mc A}(\tilde y^K)}+\norm{y^K-\tilde y^K}\\
			&\stackrel{(c)}{\le}2\norm{y^K-\tilde y^K},
		\end{aligned}
		\] 
	\fi
	\sloppy where (a) follows in view of the reverse triangle inequality, (b) since in view of Assumption~\ref{ass:COCO} the operator $\mc O(\cdot,\xi)$ is $\theta$-cocoercive and $\gamma\in(0,2\theta)$, and hence we know from \cite[Prop.~4.33]{bauschke2017correction} that $\Id(\cdot)-\gamma \mc O(\cdot,\xi)$ is $(\gamma/2\theta)$-averaged, yielding $\norm{x^{K+1}-\gamma \mc O(x^{K+1},\xi)-(\tilde x^{K+1}-\gamma \mc O(\tilde x^{K+1},\xi))}^2\le\norm{x^{K+1}-\tilde x^{K+1}}^2-\tfrac{2\theta-\gamma}{\gamma}\norm{\gamma \mc O(x^{K+1},\xi)-\gamma \mc O(\tilde x^{K+1},\xi)}^2\le\norm{x^{K+1}-\tilde x^{K+1}}^2$, and (c) since $\resolvent{J}{\gamma\mc A}(\cdot)$ is the resolvent of the maximally monotone operator $\mc A$, and hence firmly nonexpansive \cite[Cor.~23.10]{bauschke2017correction}, i.e., for all $\gamma>0$ it satisfies $\norm{\resolvent{J}{\gamma\mc A}(y^K)-\resolvent{J}{\gamma\mc A}(\tilde y^K)}^2\le\norm{y^K-\tilde y^K}^2-\norm{(y^K-\resolvent{J}{\gamma\mc A}(y^K))-(\tilde y^K-\resolvent{J}{\gamma\mc A}(\tilde y^K))}^2\le\norm{y^K-\tilde y^K}^2$.
	According to Definition~\ref{def:unif_stab}, also in this case it remains to bound $\|y^K-\tilde y^K\|$, possibly as a function of $1/s$.
	
	To this end, we note that $\|y^K-\tilde y^K\|=\|x^K-\frac{\gamma}{s} \sum_{j=1}^s  \mc O(x^K,\xi^{(j)})-(\tilde x^K-\frac{\gamma}{s} \sum_{j=1, j\neq i}^s \mc O(\tilde x^K,\xi^{(j)})-\frac{\gamma}{s} \mc O(\tilde x^K,\xi'))\|$. More generally, at the $k$-th iteration we have:
	\ifTwoColumn
\else
		\[
		\begin{aligned}
			\|y^k-\tilde y^k\|&=\norm{x^k-\frac{\gamma}{s} \sum_{j=1}^s  \mc O(x^k,\xi^{(j)})-\tilde x^k+\frac{\gamma}{s} \sum_{j=1}^s  \mc O(\tilde x^k,\xi^{(j)})}\\
			&\le\norm{x^k-\frac{\gamma}{s} \sum_{\stackrel{j=1}{j\neq i}}^s  \mc O(x^k,\xi^{(j)})-\left(\tilde x^k-\frac{\gamma}{s} \sum_{\stackrel{j=1}{j\neq i}}^s  \mc O(\tilde x^k,\xi^{(j)})\right)}\\
			&\hspace{7cm}+\norm{\frac{\gamma}{s} \mc O(x^k,\xi^{(i)})}+\norm{\frac{\gamma}{s} \mc O(\tilde x^k,\xi')}\\
			&\le\norm{x^k-\tilde x^k}+\frac{2\gamma M}{s}=\norm{\resolvent{J}{\gamma\mc A}(y^{k-1})-\resolvent{J}{\gamma\mc A}(\tilde y^{k-1})}+\frac{2\gamma M}{s}\\
			&\le\norm{y^{k-1}-\tilde y^{k-1}}+\frac{2\gamma M}{s},
		\end{aligned}
		\] 
	\fi
	where the second to last inequality follows by combining Lemma~\ref{lemma:cocoercivity} with the boundedness of the operator $\mc O(\cdot,\xi)$ in Standing Assumption~\ref{standing:boundedness}. From \cite[Prop.~4.33]{bauschke2017correction}, the main consequence of Lemma~\ref{lemma:cocoercivity} is indeed the $(\gamma/2\theta)$-averagedness of the operator $\Id(\cdot)-\tfrac{\gamma}{s}\sum_{\stackrel{j=1}{j\neq i}}^s\mc O(\cdot,\xi^{(j)})$ when $\gamma<2\theta$, which is imposed through Assumption~\ref{ass:COCO}. The last inequality then follows again due to the firm nonexpansiveness of $\resolvent{J}{\gamma\mc A}(\cdot)$.
	
	Thus, recursing over the $K$-steps by considering that i) the resolvent operator $\resolvent{J}{\gamma\mc A}(\cdot)$ is (firmly) nonexpansive, and ii) both trajectories leading to $y^K$ and $\tilde y^K$ originate from the same initial condition $x^0\in\R^n$, yields
	$\|y^K-\tilde y^K\|\le2\gamma M K/s$.
	Putting everything together, we finally obtain what follows:
	\[
	|\ell(H_s,\xi)-\ell(H_{s^{i}},\xi)|\le\frac{4\gamma M K}{s},
	\]
	which proves uniform stability with bound inverse to $s$ and linear in $K$, concluding the proof.
\end{proof}

We can finally prove the following sample complexity bound:
\begin{theorem}\label{th:FB_sample_complexity_general}
	Fix $s\ge1$ and $\delta\in(0,1)$. Given any dataset $\mc D_s\in\Xi^s$, under Assumption~\ref{ass:COCO} there exists $\varepsilon=\varepsilon(s,\delta,K)>0$ such that $x^{K+1}\in\zer_\varepsilon(\mc A+\mc B)$ with probability at least $1-\delta$, where $x^{K+1}$ is obtained by iterating Algorithm~\ref{alg:FBmod} $K$-times.
	\hfill$\square$
\end{theorem}
\begin{proof}
	As in the proof of Theorem~\ref{th:FB_sample_complexity}, upper bounding the risk $\expected{\prob{P}}{\ell(H_s,\xi)}$ provides an equivalent information on the distance of $x^{K+1}$ from a fixed point of the dynamics produced by iterating the operator $\resolvent{J}{\gamma\mc A}\circ(\Id-\gamma\mc B)$.
	Therefore, with probability at least $1-\delta$, applying the bound in \eqref{eq:exp_bound_replacement} leads to $\norm{\resolvent{J}{\gamma\mc A}(x^{K+1}-\gamma\mc B(x^{K+1}))-x^{K+1}}\le\epsilon$ with
	\ifTwoColumn
\else
		\[
		\epsilon\eqdef\frac{1}{s} \sum_{i=1}^{s}\norm{x^{K+1}-\gamma \mc O(x^{K+1},\xi^{(i)})-y^K}\\
		+\frac{4\gamma M K}{s}+(4\gamma M K+\bar\ell)\sqrt{\frac{2\ln(1/\delta)}{s}}.\notag
		\]
	\fi
	Akin to the conclusion of the proof of Theorem~\ref{th:FB_sample_complexity}, $x^{K+1}$ is then an $\varepsilon$-zero in the spirit of \eqref{eq:eps_zero}, with $\varepsilon\eqdef\epsilon/\gamma$.
\end{proof}

\begin{remark}
	Our results are not only consistent with, but also generalize existing ones on \gls{SGD} in convex optimization, since our operator-theoretic framework embraces a wider class of problems. When the cost function at hand is simply convex, \gls{SGD} is characterized by a stability bound proportional to the number of iterations performed \cite[Th.~3.8]{hardt2016train}. Consistently, in \S \ref{sec:coco} we show that Algorithm~\ref{alg:FBmod} exhibits a stability bound proportional to $K$ when the operator $\mc B$ is merely cocoercive. Conversely, if the cost is strongly convex, there is no dependency on the number of iterations \cite[Th.~3.9]{hardt2016train}, and we obtain an analogous result in \S \ref{sec_stab_strong} with strongly monotone $\mc B$.
	\hfill$\square$
\end{remark}

\subsection{Discussion on our results}\label{subsec:discussion}
The generalization properties established for Algorithm~\ref{alg:FBmod} hence provide computable bounds on the distance between its output, after a finite number of steps, and a zero of the operator $\mc A + \mc B$. As discussed next, this is particularly relevant for real-world applications where data collection through repeated experiments is either prohibitively expensive or infeasible, thus calling for the efficient use of the available, limited samples:
\begin{itemize}
	\item Reliability of the solution: Offering high confidence-type of bounds makes the obtained $x^{K+1}$ reliable \gls{wrt} any possible dataset drawn from $\Xi^s$. No matter what $\mc D_s$ is employed to run Algorithm~\ref{alg:FBmod}, the solution produced will then always be an $\varepsilon$-zero of $\mc A+\mc B$, with $\varepsilon$ vanishing as $O(s^{-1/2})$, up to a set of arbitrarily small measure $\delta$;
	\item Easy-to-compute bounds: Although \emph{a-posteriori}, i.e., once $x^{K+1}$ is revealed, the specific choice of the loss function in \eqref{eq:FBloss}, and hence the empirical risk $\hat r(A, s)$ associated, makes our bounds particularly easy to compute. Notice that for our purposes, i.e., to characterize $x^{K+1}$ as an $\varepsilon$-zero of $\mc A+\mc B$, we could even have chosen $\ell(H,\xi)=\norm{x^\star- [0_n\quad I_n] H}=\norm{x^\star-\resolvent{J}{\gamma\mc A}(x-\gamma \mc O(x,\xi))}$ for some $x^\star\in\zer(\mc A+\mc B)$, thereby obtaining similar stability bounds as in Lemmas~\ref{lemma:FB_unif_stab} and \ref{lemma:FB_unif_stab_general}. However, computing $\varepsilon$ in this case is prohibitive unless one knows in advance such an $x^\star$, making the zero-finding problem meaningless;
	\item Distribution-free: Along the line of results based on non-parametric statistics and learning theory, such as \gls{PAC} learning \cite{vidyasagar2013learning} or scenario approach theory \cite{campi2018introduction}, we further emphasize that we only require the data to be \gls{iid}. Our bounds are thus derived without imposing any specific condition on the unknown probability distribution $\prob{P}$ underlying available samples. This provides robust guarantees regardless of how skewed or heavy-tailed the true distribution might be.
\end{itemize}


\section{Finite-sample guarantees for equilibrium seeking in stochastic Nash games}\label{sec:SGNEP}
As a suitable motivating application, we now illustrate how a standard \gls{NE} seeking algorithm in a stochastic regime fits the framework considered in this paper, and hence benefit the sample complexity bounds developed in \S \ref{sec:FB}.

\subsection{A stochastic Nash equilibrium problem}
We consider a set of $N$ self-interested agents, indexed by $\mc I\eqdef\{1,\ldots,N\}$, where each of them controls a strategy $x_i\in\Omega_i\subseteq\R^{n_i}$. The goal is then to simultaneously solve a collection of mutually coupled stochastic optimization problems:
\begin{equation}\label{eq:sNEP}
	\forall i\in\mc I:\underset{x_i\in\R^{n_i}}{\textrm{min}}~\expected{\prob{P}}{J_i(x_i,\bs x_{-i},\xi)}+g_i(x_i),
\end{equation}
where each $J_i:\R^n\times\R^d\to\R$ is some measurable function so that, together with $g_i:\R^{n_i}\to\R$, constitutes the cost of the $i$-th agent, where $n\eqdef\sum_{i\in\mc I}^{}n_i$ and $\bs x_{-i}\eqdef\col((x_j)_{j\in\mc I\setminus\{i\}})\in\R^{n-n_i}$. The vector $\xi\in\R^d$ denotes, instead, some uncertainty affecting each $J_i$, and corresponds to a random variable taking values in some $\Xi\subseteq\R^d$ according to a possibly unknown probability distribution $\prob{P}$. Each function $g_i(\cdot)$ typically represents a nonsmooth term modeling local constraints via an additive indicator function $\iota_{\Omega_i}(\cdot)$. 
Thus, the cost function depends on the local variable $x_i$, on the decision of the other agents $\bs x_{-i}$, and on the random variable $\xi$.

Let $\mathbb{J}_i(x_i,\bs x_{-i})\eqdef\expected{\prob{P}}{J_i(x_i,\bs x_{-i},\xi)}+g_i(x_i)$ and $\Omega\eqdef\prod_{i\in\mc I} \Omega_i$.  Aiming at solving the collection of optimization problems in \eqref{eq:sNEP} simultaneously, thereby resulting into a \gls{SNEP} \cite{lei2022stochastic}, the solution concept of our interest is formalized next:
\begin{definition}[\textup{\hspace{-.02cm}{\cite[Definition 2.1]{ravat2011characterization}} Stochastic Nash equilibrium}]\label{def:SNE}
	A collective decision vector $\bs x^\star\in\R^n$ is a \gls{SNE} if, for all $i\in\mc I$,
	\[
		\mathbb{J}_i(x^\star_i,\bs x^\star_{-i})\le\underset{y_i\in\R^{n_i}}{\normaltext{\textrm{inf}}}~\mathbb{J}_i(y_i,\bs x^\star_{-i}).\vspace{-.6cm}
	\]
	\hfill$\square$
\end{definition}
In words, an \gls{SNE} is a set of strategies where no agent can decrease its cost function by unilaterally deviating from its decision. Moreover, we will call $\varepsilon$-\gls{SNE} any collective strategy profile $\bs x^\star\in\Omega$ that satisfies, for all $i\in\mc I$, $\mathbb{J}_i(x^\star_i,\bs x^\star_{-i})\le\textrm{inf}_{y_i\in\R^{n_i}}~\mathbb{J}_i(y_i,\bs x^\star_{-i})+\varepsilon$, for some $\varepsilon>0$.
We now make some standard assumptions on the data of the \gls{SNEP} at hand \cite{franci2020distributed}:
\begin{standing}[\textup{Local cost function}]\label{standing:cost}
	The following conditions hold true for all $i\in\mc I$:
	\begin{enumerate}
		\item[(i)] The function $x_i\mapsto g_i(x_i)$ is lower semicontinuous and convex with nonempty, compact, and convex domain, $\mathrm{dom}(g_i)\subseteq\R^{n_i}$;
		\item[(ii)] For all $\bs x_{-i}\in\R^{n-n_{i}}$ and $\xi\in\Xi$, the function $x_i\mapsto J_i(x_i,\bs x_{-i},\xi)$ is convex, Lipschitz continuous, and continuously differentiable. For all $x_i\in\R^{n_i}$ and $\bs x_{-i}\in\R^{n-n_{i}}$, the function $\xi\mapsto J_i(x_i,\bs x_{-i},\xi)$ is measurable with integrable Lipschitz constant $\kappa_i(\bs x_{-i},\xi)>0$ \gls{wrt} $\xi$;
		\item[(iii)] For all $\bs x_{-i}\in\R^{n-n_{i}}$, $x_i\mapsto \expected{\prob{P}}{J_i(x_i,\bs x_{-i},\xi)}$ is convex and continuously differentiable. \hfill$\square$
	\end{enumerate}
\end{standing}
The conditions postulated in Standing Assumption~\ref{standing:cost} guarantee the existence of an \gls{SNE}, while they are not sufficient to establish uniqueness \cite{ravat2011characterization}. Thus, the derivation of some iterative \gls{SNE} seeking algorithms able to compute an equilibrium is typically achieved by recasting the \gls{SNEP} in \eqref{eq:sNEP} as a zero-finding problem, i.e., by looking for some $\bs x^\star$ satisfying:
\begin{equation}\label{eq:KKT}
	\text{For all } i \in\mc I:\bs 0_{n_i}\in\expected{\prob{P}}{\nabla_{x_i} J_i(x^\star_i,\bs x^\star_{-i},\xi)}+\partial g_i(x^\star_i),
\end{equation}
where, in view of Standing Assumption~\ref{standing:cost}.(ii), we are entitled to exchange the expected value and the gradient. We recall that the inclusions above can be derived by imposing the \gls{KKT} conditions to each problem \eqref{eq:sNEP}.  By introducing $\mathbb{F}(\bs x)\eqdef\expected{\prob{P}}{F(\bs x,\xi)}$ with $F(\bs x,\xi)\eqdef\col((F_i(x_i,\bs x_{-i},\xi))_{i\in\mc I})$ and $F_i(x_i,\bs x_{-i},\xi)\eqdef\nabla_{x_i} J_i(x_i,\bs x_{-i},\xi)$, and 
$\partial g(\bs x)\eqdef\col((\partial g_i(x_i))_{i\in\mc I})$,
conditions in \eqref{eq:KKT} can be represented in compact form by operators:
\[
	\mathcal{A}:
		\bs x\mapsto
		\partial g(\bs x), \; \mathcal{B}:
		\bs x \mapsto
		\mathbb{F}(\bs x).
\] 
Note that the conditions in \eqref{eq:KKT} require the evaluation of the expected value $\expected{\prob{P}}{\cdot}$, a quantity that is impossible to compute in case the distribution $\prob{P}$ of the random variable $\xi$ is not available, as in our framework. Therefore, several data-driven approximation schemes for the evaluation of the pseudogradient mapping, i.e., the operator $\mc B$ above, have been proposed in the literature \cite{koshal2013,franci2020distributed,franci2021distributed,franci2021stochastic,fabiani2022stochastic}. 
However, we take a practical perspective by assuming that the agents have collectively available a finite set of \gls{iid} realizations for $\xi$ coming from, e.g., historical data, gathered into $\mc D_s$. This is particularly relevant for instance in autonomous driving \cite{palanisamy2020multi}, traffic coordination \cite{liu2023unified}, as well as smart grids \cite{xu2020multi} and demand response management \cite{motalleb2018real}. A key example from the latter domain will be analyzed in \S \ref{subsec:SNEP}.

We now report conditions on $\mathbb{F}(\cdot)$ ensuring, among the other, the uniqueness of the \gls{SNE} associated to the \gls{SNEP} in \eqref{eq:sNEP}:
\begin{standing}\label{standing:pseudogradient}
		For all $\xi\in\Xi$, the mapping $\bs x\mapsto F(\bs x,\xi)$ is $\mu_F$-strongly monotone and $\kappa_F$-Lipschitz continuous, with $\mu_F\le\kappa_F$, and $\gamma\in(0,2\mu_F/\kappa_F^2)$. Given any $\xi\in\Xi$, there exists $M_F>0$ so that $\|F(\bs x,\xi)\|\le M_F$ for all $\bs x\in\R^n$.
	\hfill$\square$
\end{standing}

\begin{algorithm}[t]
	\caption{Data-driven proximal gradient method}\label{alg:proximal}
	
	\smallskip
	\textbf{Initialization:} Samples $\mc D_s$, $x_i^0 \in\R^{n_i}$ for all $i\in\mc I$
	\smallskip
	
	\textbf{Iteration} $k\in\N_0$\textbf{:} Agent $i$ receives $\bs x_{-i}^k$, then updates
	\[
	x_i^{k+1}=\prox_{\gamma g_i}\left(x_i^k-\frac{\gamma}{s} \sum_{j=1}^s  F_i(x_i^k,\bs x_{-i}^k,\xi^{(j)})\right)
	\]
\end{algorithm}

Our analysis focuses on the distributed, proximal-like procedure reported in Algorithm~\ref{alg:proximal}, which can be derived by starting from the agent-wise fixed-point \gls{FB} iteration, i.e., for all $i\in\mc I$:
\[
	x_i^\star=(I_{n_i}-\gamma\partial g_i)^{-1}(x_i^\star+\gamma\expected{\prob{P}}{\nabla_{x_i} J_i(x^\star_i,\bs x^\star_{-i},\xi)}).
\]
From Standing Assumption \ref{standing:cost} and \cite[Prop.~16.44]{bauschke2017correction}, we have $(I_{n_i}-\gamma\partial g_i)^{-1}=\prox_{\gamma g_i}$, leading to the exact version of Algorithm~\ref{alg:proximal}. Similar to \eqref{eq:approxB}, the expected value can then be replaced through a data-driven approximation of the local operator $\expected{\prob{P}}{F_i(x_i,\bs x_{-i},\xi)}$ by averaging $F_i(\cdot,\cdot,\xi^{(j)})$ at every iteration over the available $s$-samples.
 Also in this case we tacitly assume each $F_i(\cdot,\xi)$ being an unbiased operator, i.e., for all $\bs x\in\R^n$, $\expected{\prob{P}}{F_i(\bs x,\xi)}=\mathbb{F}_i(\bs x)$, for all $i\in\mc I$. Moreover, to simplify the analysis, we have adopted the same, constant learning rate $\gamma>0$ for all the agents, albeit a generalization to local, possibly time-varying $\{\gamma_{k_i}\}_{i\in\mc I}$ is doable. A further generalization consists in considering local dataset $\mc D_{s^i}$. 

Under the conditions postulated in Standing Assumptions~\ref{standing:cost} and \ref{standing:pseudogradient}, which are widely employed in the algorithmic game theory literature for stochastic setting \cite{ravat2011characterization, franci2021stochastic}, we note that Algorithm~\ref{alg:proximal} exhibits almost sure convergence to the \gls{SNE} in case of an increasing batch size $s=s_k\to\infty$, along with few other assumptions \cite[Th.~1]{franci2020distributed}. In a data-limited context, however, our goal is instead to investigate how far we can get from the unique \gls{SNE} by running Algorithm~\ref{alg:proximal} for $K$ iterations. Specifically, we want to characterize, in a probabilistic sense, the resulting $\bs x^{K+1}$ as an $\varepsilon$-\gls{SNE} of the \gls{SNEP} in \eqref{eq:sNEP}.


\subsection{Data-driven certificates for $\varepsilon$-\gls{SNE}}
First, we need to recast Algorithm~\ref{alg:proximal} into the framework of algorithmic stability by starting from its compact reformulation:
\begin{equation}\label{eq:proximal_compact}
	\bs x^{k+1}=\bs{\prox}_{\gamma g}\left(\bs x^k-\frac{\gamma}{s} \sum_{j=1}^s  F(\bs x^k,\xi^{(j)})\right),
\end{equation}
where, with some abuse of notation, $\bs{\prox}_{\gamma g}(\cdot)$ is meant to be applied agent-wise with $g_i$ in place of $g$. For the sake of our analysis, \eqref{eq:proximal_compact} can be equivalently rewritten by means of an additional variable $\bs y\in\R^n$ as follows:
\begin{equation}\label{eq:proximal_compact_augmented}
	\left\{
	\begin{aligned}
		&\bs y^k=\bs x^k-\frac{\gamma}{s} \sum_{j=1}^s  F(\bs x^k,\xi^{(j)})\\
		&\bs x^{k+1}=\bs{\prox}_{\gamma g}\left(\bs y^k\right).
	\end{aligned}
	\right.
\end{equation}
In fact, due to the properties of the proximity operator we~have: 
\[
\left\{
\begin{aligned}
	&\bs y^\star=\bs x^\star-\gamma \mathbb F(\bs x^\star)\\
	&\bs x^\star=\bs{\prox}_{\gamma g}\left(\bs y^\star\right)=\bs{\prox}_{\gamma g}\left(\bs x^\star-\gamma \mathbb F(\bs x^\star)\right),
\end{aligned}
\right.
\]
i.e., a fixed point of the dynamics in \eqref{eq:proximal_compact_augmented}, in case $\mathbb F(\cdot)$ was available and in place of the average $\frac{1}{s} \sum_{j=1}^s  F(\bs x^k,\xi^{(j)})$, equivalently produces an \gls{SNE} $\bs x^\star$ as subvector, thereby linking the augmented dynamics  \eqref{eq:proximal_compact_augmented} with the conditions in \eqref{eq:KKT}.
By making a parallelism with \S \ref{subsec:alg_stab}, starting from $\bs x^0\in\R^n$ we will let coincide with our hypothesis $H_s$ the output provided by iterating \eqref{eq:proximal_compact_augmented} $K$-times, i.e., $\bs \omega^K\eqdef\col(\bs y^K,\bs x^{K+1})\in\R^{2n}$.


The considerations above suggest us to choose the following function as a candidate loss associated to a hypothesis $H$:
\begin{equation}\label{eq:sneploss}
	\begin{aligned}
		\ell(H,\xi)&=\norm{[0_n\quad I_n] H-\gamma F([0_n\quad I_n] H,\xi)-[I_n\quad0_n]H}\\
		&=\norm{
			\bs x-\gamma F(\bs x,\xi)-\bs y}.
	\end{aligned}
\end{equation}
We can then establish what follows:
\begin{lemma}\label{lemma:unif_stab_3}
	With $\tau_F\eqdef\sqrt{1-\gamma(2\mu_F-\gamma\kappa_F^2)}<1$, Algorithm~\ref{alg:proximal} possesses $(2\gamma M_F(1+\tau_F)/(s(1-\tau_F)))$-uniform stability \gls{wrt} the loss function $\ell(H,\xi)$ in \eqref{eq:sneploss}.
	\hfill$\square$
\end{lemma}
\begin{proof}
	The proof follows exactly the same steps as per the proof of Lemma~\ref{lemma:FB_unif_stab}, once noted that Standing Assumptions~\ref{standing:cost} and \ref{standing:pseudogradient} together allow \eqref{eq:proximal_compact_augmented} to satisfy the conditions postulated in Assumption~\ref{ass:SMON}.
\end{proof}
As a main consequence of the uniform stability exhibited by Algorithm~\ref{alg:proximal}, we obtain the following key result:
\begin{theorem}\label{th:sample_complexity_snep}
	Fix $s\ge1$ and $\delta\in(0,1)$. Given any dataset $\mc D_s\in\Xi^s$, there exists $\varepsilon=\varepsilon(s,\delta)>0$ such that $\bs x^{K+1}$ is an $\varepsilon$-\gls{SNE} of the \gls{SNEP} in \eqref{eq:sNEP} with probability at least $1-\delta$, where $\bs x^{K+1}$ is obtained by iterating Algorithm~\ref{alg:proximal} $K$-times.
	\hfill$\square$
\end{theorem}
\begin{proof}
	Once observed that measuring the gap between consecutive iterations on the decision variable $\bs x$ through $\|\bs{\prox}_{\gamma g}\left(\bs x^{k+1}-\gamma \mathbb F(\bs x^{k+1})\right)-\bs x^{k+1}\|$ provides an equivalent information on the distance to a fixed point of the dynamics $\bs x^{k+1}=\bs{\prox}_{\gamma g}\left(\bs x^k-\gamma \mathbb F(\bs x^k)\right)$, the rest of the proof mimics the steps performed for deriving Theorem~\ref{th:FB_sample_complexity}, thus ending up to upper bounding $\|\bs{\prox}_{\gamma g}\left(\bs x^{K+1}-\gamma \mathbb F(\bs x^{K+1})\right)-\bs x^{K+1}\|$ with the risk associated to \eqref{eq:sneploss}.
\end{proof}

Theorem~\ref{th:sample_complexity_snep} then offers a finite-sample probabilistic certificate, in the form of a distance $\varepsilon$ from the unique \gls{SNE}, characterizing the output of Algorithm~\ref{alg:proximal} obtained after arbitrary $K$ iterations. This is crucial in real-world applications where available dataset may even be large, yet limited, and whose samples are only required to be \gls{iid}, with no restrictions on the probability $\prob{P}$ according to which they are distributed.

\begin{remark}
	The uniform stability established in Lemma~\ref{lemma:FB_unif_stab_general} offers an opportunity to extend our bounds for $\varepsilon$-\gls{SNE} to \glspl{SGNEP}.
	 Here, the agents aim at solving mutually coupled optimization problems with shared constraints $h:\R^n\to\R^m$, i.e.,
	\[
	\forall i \in \mc I:
	\underset{x_i \in \Omega_i}{\mathrm{min}}~\mathbb{J}_i\left(x_i, \bs{x}_{-i}\right)~\mathrm{ s.t. }~h(x_i,\bs x_{-i})\leq0.
	\]
Also in this case, an equilibrium can be seen as a zero of the sum of the following two operators \cite{yi2019operator,belgioioso2018projected,franci2021stochastic}:
\[
\begin{aligned}
& \mathcal{A}:\begin{bmatrix}
\bs x \\
\lambda
\end{bmatrix}\mapsto\begin{bmatrix}
\partial g(\bs x) \\
\mathrm N_{\mathbb{R}_{\geq 0}^m}(\lambda)
\end{bmatrix}+\begin{bmatrix}
\nabla h(\bs x)^\top \lambda \\
-h(\bs x)
\end{bmatrix}, \ \mathcal{B}:\begin{bmatrix}
\bs x \\
\lambda
\end{bmatrix} \mapsto\begin{bmatrix}
\mathbb{F}(\bs x) \\
0
\end{bmatrix},
\end{aligned}\] 
where, typically, the (extended) mapping $\mc B$ is cocoercive \cite[Lemma~1]{belgioioso2018projected}.
Following similar arguments as in Lemma~\ref{lemma:FB_unif_stab_general}, we can then prove $K$-dependent uniform stability of a tailored data-driven \gls{FB} for \glspl{SGNEP} and establish a bound analogous to the one in Theorem~\ref{th:FB_sample_complexity_general}.
	\hfill$\square$
\end{remark}
\section{Numerical experiments}\label{sec:simulations}
We now verify our theoretical bounds on several numerical examples. All simulations are run in MATLAB on a laptop with an Apple M2 chip featuring an 8-core CPU and 16 GB~RAM.

\subsection{Plug-in electric vehicles charging coordination}\label{subsec:SNEP}
To test the bound established in Theorem~\ref{th:sample_complexity_snep}, a particular instance of that in Theorem~\ref{th:FB_sample_complexity}, we adopt a stochastic version of a classic charging coordination problems for \glspl{PEV}, sketched next for completeness.
Specifically, we consider a set of $N$ \glspl{PEV} populating a distribution grid, indexed by $\mc I=\{1,\ldots,N\}$, where each agent aims to determine a day-ahead charging schedule subject to a stochastic cost of electricity \cite{liu2015bidding} and few other fees. To this end, each \gls{PEV} directly controls variable $x_i\in\Omega_i\subseteq\R^T$ corresponding to the energy injection over a discrete interval $\mc T=\{1,\ldots,T\}$. Each user then aims at minimizing a private cost in the form: 
\begin{equation}\label{eq:SNEP_local_cost}
	J_i(x_i,\sigma(\bs x)) = \|x_i\|^2_{Q_i}+c_i^\top x_i + \xi^\top x_i  + \|\sigma(\bs x) - \bar{\sigma} \|^2_P,
\end{equation} 
where $\|x_i\|^2_{Q_i}+c_i^\top x_i$ models the $i$-th battery degradation cost, with $0\prec Q_i\in\R^{T\times T}$ and $c_i\in\R^T$, while $\xi\in\R^{T}$ denotes the stochastic day-ahead price of energy, supported over some $\Xi\subset\R^{T}$ with unknown probability distribution $\prob{P}$. Moreover, $\sigma(\bs x)$ represents the aggregate demand of the overall population of \glspl{PEV}, usually defined as $\sigma(\bs x) = \tfrac{1}{N} \sum_{i=1}^{N} x_i \in \Omega$, $\Omega=\prod_{i\in\mc I}\Omega_i$, whose deviation from some reference signal $\bar{\sigma}\in\mathbb{R}^T$ is penalized through $0\prec P \in \R^{T\times T}$. To complete the definition of our \gls{SNEP}, we let
$
	\Omega_i=\set{x_i\in[0, \bar x_i]^T}{\bsone_T^\top x_i \ge \zeta_i}.
$
While each $\bar x_i\ge0$ limits the nonnegative power injection at every interval, $\zeta_i \geq0$ forces a minimum level of charging amount over the time horizon $\mc T$ for the $i$-th user satisfaction.

\begin{table}[tb]
	\caption{Simulation parameters -- \S \ref{subsec:SNEP}}
	\label{tab:sim_val}
	\centering
	\begin{tabular}{lll}
		\toprule
		Param.  & Description   & Value \\
		\midrule
		$T$ & Time interval & $14$\\
		$N$ & Number of \glspl{PEV} & $20$\\
		$Q_i$ & Weight matrix - Battery degradation & $\sim \mc U(0.002,0.008) \cdot I_T$\\
		$c_i$ & Weight vector - Battery degradation & $\sim \mc U(0.02,0.075) \cdot \bsone_T$\\
		$P$ & Weight matrix -- Deviation from $\bar\sigma$ & $I_{168}$\\
		$\bar\sigma$ & Aggregate reference signal & $\bsone_T$\\
		$\bar x_i$ & Power injection cap & $2.5$\\
		$\zeta_i$ & Minimum charging amount & $\sim \mc U(12,18)$\\
		$\mu_F$&Strong monotonicity constant&$0.0127$\\
		$\kappa_F$&Lipschitz constant&$0.1159$\\
		$M_F$&Upper bound -- $\|F(\bs x,\xi)\|$&$39.2192$\\
		\bottomrule
	\end{tabular}
\end{table}

Our numerical analysis is conducted with the values reported in Tab.~\ref{tab:sim_val}, where $\mu_F$ and $\kappa_F$ have been obtained analytically in view if the quadratic structure of the \gls{SNEP} at hand, while $M_F$ has been computed numerically. 
For the data-driven representation of the variable $\xi$ affecting each cost in \eqref{eq:SNEP_local_cost}, we have collected from \cite{entsoe} ten years of day-ahead energy prices (in \euro/kWh) with granularity of one hour. In particular, the price data refer to the period January 5, 2015--December 31, 2024, i.e., since the platform \cite{entsoe} was launched, for a total of $3649$ samples. As common in these type of coordination problems in which the flexible \glspl{PEV} demand is usually analyzed during off-peak periods, we have focused on the interval 0:00am--1:00pm.

For comparison purposes only, we compute the unique \gls{SNE} $\bs x^\star$ through a simplified version of \cite[Alg.~1]{franci2020distributed}. Besides turning each proximal operator $\prox_{\gamma g_i}(\cdot)$ into a projection mapping $\proj_{\Omega_i}(\cdot)$---the same happens to Algorithm~\ref{alg:proximal} indeed---at every iteration, \cite[Alg.~1]{franci2020distributed} relies on $10 k$ \gls{iid} samples drawn from a surrogate representation of $\Xi$ obtained by simply taking the minimum and maximum value over $\mc T$ from the pool of $3649$ available realizations. Note that, in the practical setting considered, the need of leveraging a limited dataset clashes with the theoretical conditions for exact convergence to the \gls{SNE} required by \cite[Alg.~1]{franci2020distributed}, which under our choice of selecting $10$ \gls{iid} samples per step would run for at most $364$ iterations, thereby motivating our quest for finite sample guarantees.

\begin{figure}
	\centering
	\ifTwoColumn
\else
		\includegraphics[width=.6\columnwidth]{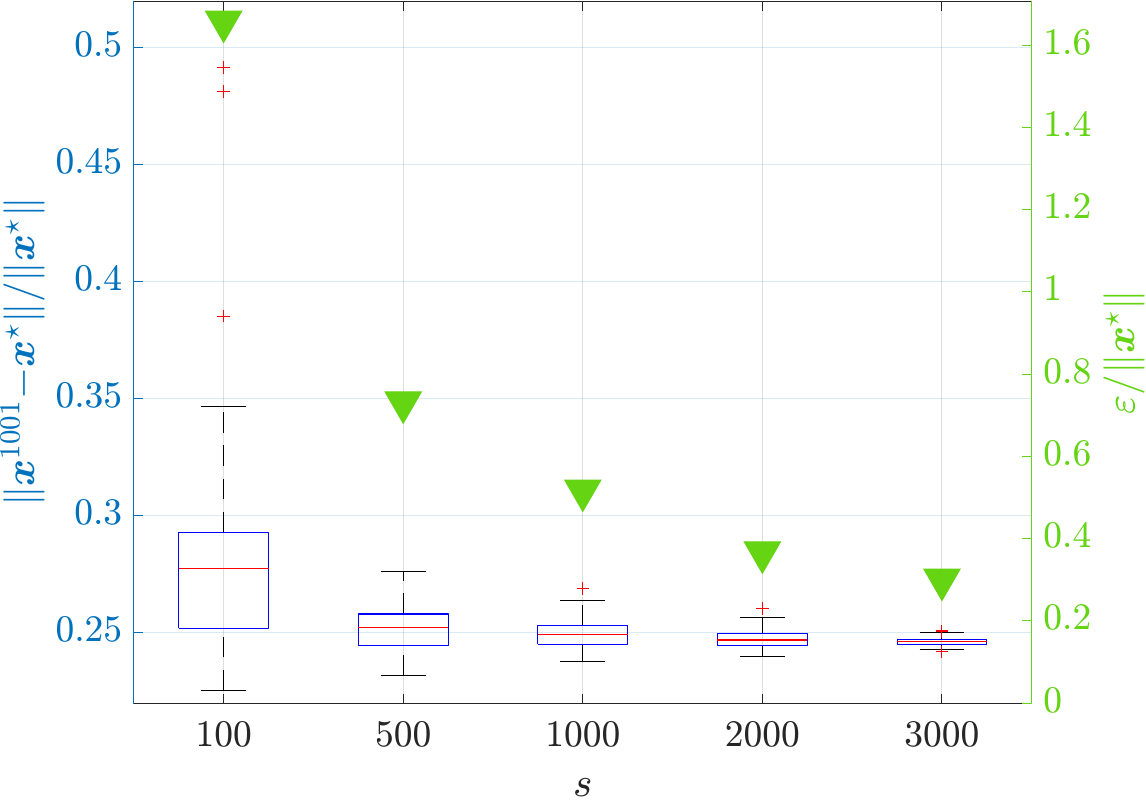}
	\fi
	\caption{Left y-axis: Box plots capturing the relative approximation error $\|\bs x^{1001}-\bs x^\star\|$ produced by $K=1000$ iterations of Algorithm~\ref{alg:proximal} over $50$ trials considering different dataset size. Right y-axis: The resulting averaged bounds $\varepsilon$ (green down-pointing triangles) offered by Theorem~\ref{th:sample_complexity_snep} with $\delta=0.05$.}
	\label{fig:testS_fixedK_t50}
\end{figure}
\begin{figure}
	\centering
	\ifTwoColumn
\else
		\includegraphics[width=.6\columnwidth]{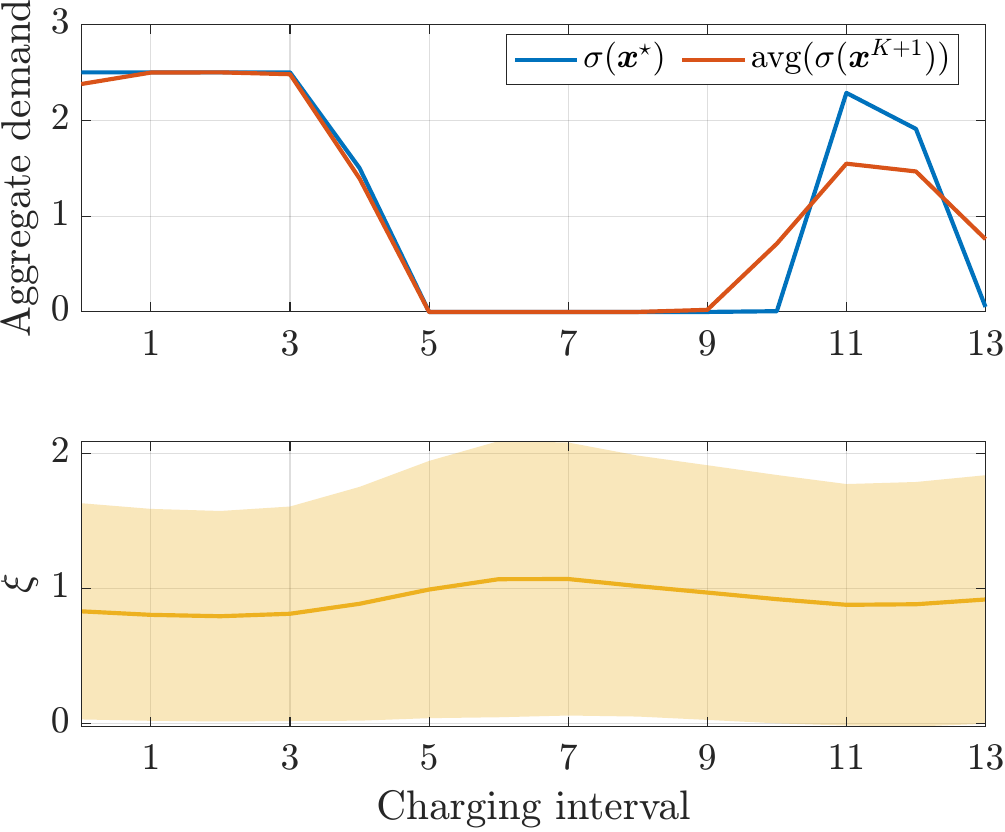}
	\fi
	\caption{Top: Aggregate demand $\sigma(\bs x^\star)$ (blue line) and the approximated one related to $\bs x^{1001}$ obtained by running Algorithm~\ref{alg:proximal} with $s=3000$, averaged over $50$ trials (red line). Bottom: Mean value (yellow line) and standard deviation (yellow shaded area) of the energy price $\xi$ over $\mc T$.}     
	\label{fig:testS_fixedK_aggdemand}
\end{figure}

First, we fix the number of iterations $K=1000$ and consider $s\in\{100, 500, 1000, 2000, 3000\}$ to compare the theoretical bound $\varepsilon$ provided by Theorem~\ref{th:sample_complexity_snep} with the actual distance $\|\bs x^{1001}-\bs x^\star\|$, both terms scaled by $\|\bs x^\star\|$. Specifically, we perform $50$ trials for each of the $s$ samples drawn from the pool of available $3649$. In Fig.~\ref{fig:testS_fixedK_t50}, where the box plots refer to the left y-axis, while the green down-pointing triangles to the right one, we report the numerical results obtained with confidence level $\delta=0.05$ and learning rate $\gamma=0.02$, which produces $\bar\ell=24.3852$. Although the performance in terms of \gls{SNE} approximation is comparable, running Algorithm~\ref{alg:proximal} with a small number of samples produces widely spread results with several outliers. However, this effect diminishes as the dataset size increases. This trend is consistently reflected in the theoretical bound $\varepsilon$. Specifically, while the average value of $\varepsilon(100,0.05)$ over $50$ trials appears loose, increasing the dataset size results in tighter bounds. In particular, we can claim that the solution produced by Algorithm~\ref{alg:proximal}, run for $K=1000$ iterations with $3000$ samples, produces an error in the approximation of $\bs x^\star$ of at most $30\%$ with probability of at least $95\%$. The top plot in Fig.~\ref{fig:testS_fixedK_aggdemand} quantifies such a statement by comparing the aggregate \glspl{PEV} demand at the \gls{SNE} and its approximation, averaged over the $50$ trials, while the bottom plot reports the mean and standard deviation of the stochastic energy price. 

\begin{figure}
	\centering
	\ifTwoColumn
\else
		\includegraphics[width=.6\columnwidth]{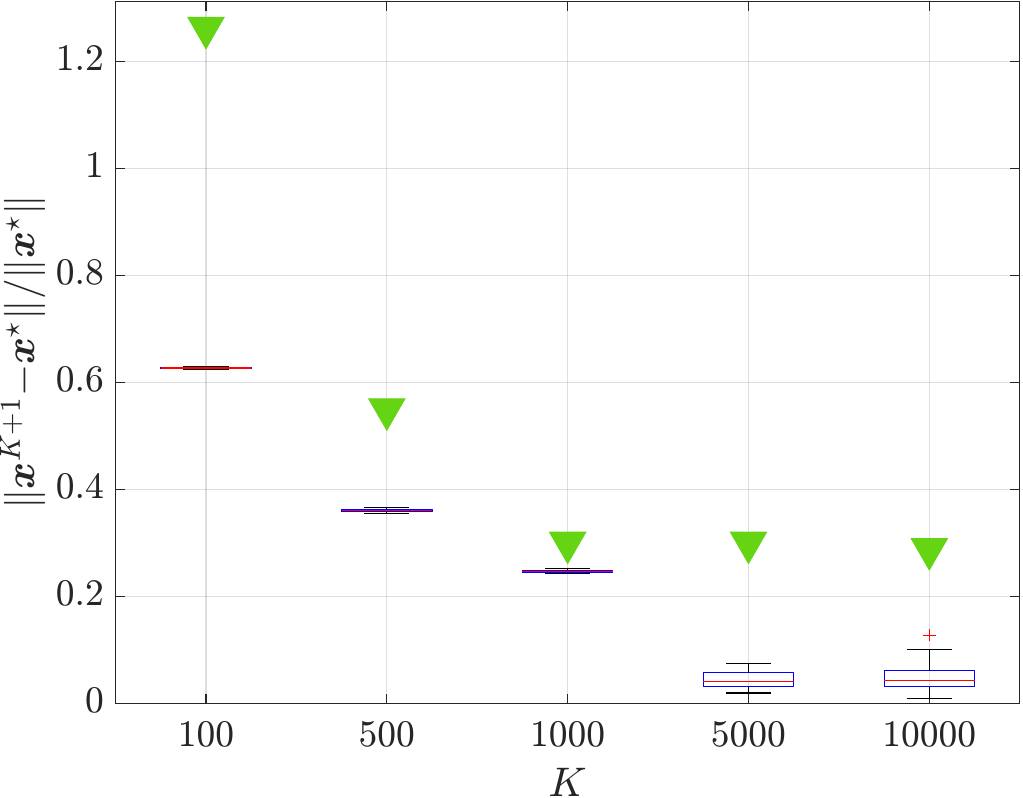}
	\fi
	\caption{Box plots capturing the relative approximation error $\|\bs x^{K+1}-\bs x^\star\|$ produced by different run $K$ of Algorithm~\ref{alg:proximal} over $50$ trials with fixed number of samples. The green down-pointing triangles denote the resulting averaged bounds $\varepsilon$ offered by Theorem~\ref{th:sample_complexity_snep} with $\delta=0.05$.}
	\label{fig:testK_fixedS_t50}
\end{figure}
\begin{figure}
	\centering
	\ifTwoColumn
\else
		\includegraphics[width=.6\columnwidth]{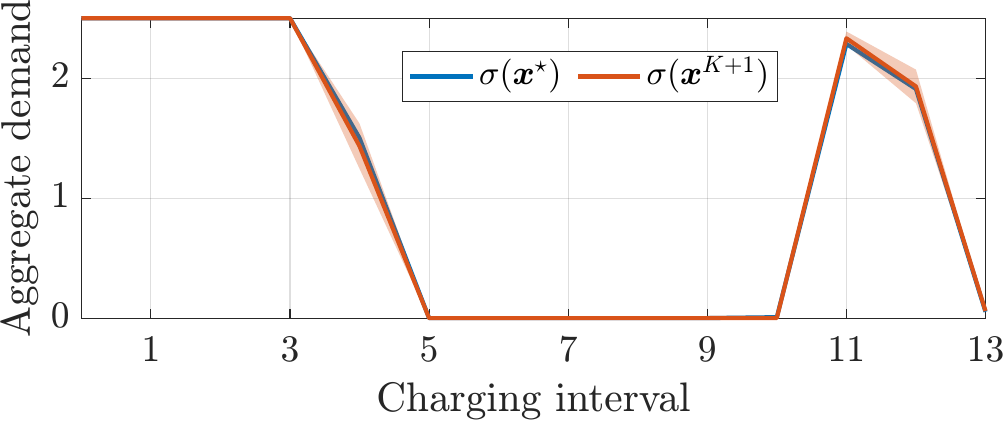}
	\fi
	\caption{Aggregate demand $\sigma(\bs x^\star)$ (blue line) and its approximation through $\bs x^{5001}$, obtained by running Algorithm~\ref{alg:proximal} with $s=3000$. The solid red line denotes the mean of $\sigma(\bs x^{5001})$, while the shaded red area the associated standard deviation.}     
	\label{fig:testK_fixedS_aggdemand}
\end{figure}

Successively, we investigate the behaviour of Algorithm~\ref{alg:proximal} for a fixed number of samples $s=3000$, $\delta=0.05$, $\gamma=0.02$, and varying $K\in\{100, 500, 1000, 5000, 10000\}$. Quite interestingly, from Fig.~\ref{fig:testK_fixedS_t50} we observe that the theoretical bound established in Theorem~\ref{th:sample_complexity_snep} reveals an intrinsic dependency on $K$ itself. In fact, considering a limited number of iterations produces an averaged bound that is substantially ``dominated'' by the empirical error $\sum_{j=1}^{s}\norm{\bs x^{K+1}-\gamma F(\bs x^{K+1},\xi^{(j)})-\bs y^K}$, which progressively reduces its impact with larger values of $K$. In addition, note that with $K\in\{5000,10000\}$ one obtains a slight dispersion related to the relative approximation error computed---the resulting aggregate behaviour is compared in Fig.~\ref{fig:testK_fixedS_aggdemand}. This numerical evidence aligns with the necessity of imposing conditions on the approximation error variance for exact \gls{SNE} computation---see, for instance, \cite[Ass.~2(f), Ass.~3(d)]{koshal2013}, and \cite[Ass.~9, Prop.~1, Ass.~19]{franci2021stochastic}.
\begin{remark}
		Our certificates are not expected to be tight. Probabilistic bounds based on stability arguments, indeed, work under weak assumptions but tend to be conservative due to their worst-case nature. Note that the gap observed in Fig.~\ref{fig:testS_fixedK_t50} and \ref{fig:testK_fixedS_t50} is comparable to the one obtained with similar probabilistic bounds \cite{schildbach2014scenario,fabiani2020robustness,fabiani2022probabilistic}. Nevertheless, sharpening strategies based on data-dependent bounds \cite{shalev2010learnability,kuzborskij2018data} or sensitivity refinements \cite{deng2021toward} can be likewise adopted.
		\hfill$\square$
\end{remark}

\subsection{An academic example}\label{subsec:academic_example}
We now test the bound obtained in \S \ref{sec:coco}, explicit function of the iteration index $K$, on a generic optimization problem which also represent a typical example for \glspl{VI} \cite{facchinei2003finite}, thus naturally fitting our operator splitting setup. Let us then consider the following quadratic optimization problem:
\[
\underset{x\in\mc X}{\textrm{min}}~q^\top x+\frac{1}{2}x^\top Px,
\]
with $\mc X\subseteq\R^n$. The data vectors/matrices are $q\in\R^n$ and $P\in\R^{n\times n}$, $P\succcurlyeq0$. Similar to \eqref{eq:KKT}, the problem of finding an optimizer $x^\star$ for the above can be rewritten as:
\[
	\bs 0_{n}\in \underbrace{q+\bar Px}_{\reqdef\mc B(x)}+\underbrace{\mathrm N_{\mc X}(x)}_{\reqdef\mc A(x)},
\]
where $\bar P$ denotes the symmetric part of $P$, i.e., $\bar P\eqdef\frac{1}{2}(P+P^\top)\in\bbS^n_{\succcurlyeq 0}$. In view of \cite[Cor.~18.16]{bauschke2017correction}, note that the operator $\mc B$ is $(1/\lambda_{\textrm{max}}(\bar P))$-cocoercive, and hence $\gamma\in(0,2/\lambda_{\textrm{max}}(\bar P))$. 

To exemplify the case in which the entries of the matrix $M$ are uncertain, we generate perturbed random matrices $\mc D_s=\{\bar P^{(1)},\ldots,\bar P^{(s)}\}$, $\bar P^{(i)}\eqdef\bar P+\diag(\xi^{(i)})\in\bbS^n$ for all $i\in\{1,\ldots,s\}$, run Algorithm~\ref{alg:FBmod} and verify the result established in Theorem~\ref{th:FB_sample_complexity_general} numerically. Specifically, we consider $n=10$, and generate $\bar P$ as $\bar P=Q\Lambda Q^\top$, where $Q\in\R^{10\times10}$ denotes the matrix of the QR decomposition of a random matrix with normally distributed entries, while the diagonal $\Lambda\in\R^{10\times10}$ has elements linearly spaced in $[0, 1]$. This yields a cocoercivity constant $1/\lambda_{\textrm{max}}(Q\Lambda Q^\top)=1$. Note that, although symmetric with each element of the random parameter $\xi$ sampled as $\mc N(0,0.5)$, not all the produced samples $\bar P^{(i)}$ are guaranteed to be positive semidefinite. In addition, each element of $q$ is chosen at random according to $\mc N(0,0.5)$, along with the constraint set $\mc X=[0,a]^{10}$, with $a\sim \mc U(0,2)$.
By considering $\delta=0.05$, all the other parameters involved in Theorem~\ref{th:FB_sample_complexity_general} and not explicitly mentioned have been estimated numerically.

\begin{table}[tb]
	\caption{Comparison results--Varying $K$}
	\label{tab:results_varyingK}
	\centering
	\ifTwoColumn
\else
		\begin{tabular}{cccccc}
			$K$ & $100$ & $500$ & $1000$ & $5000$ & $10000$\\
			\midrule
			$\textrm{avg}(\Delta x_K)\!\times\!10^{-3}$ & $443.8$ & $18.2$ &   $1.1$  &  $1.3$  &  $1.4$\\
			$\textrm{avg}(\varepsilon)$ & $0.49$  &  $0.85$  &  $1.29$ &   $2.37$ &   $3.67$\\
			\bottomrule
		\end{tabular}
	\fi
\end{table}
In this case we are interested in investigating how the dependency on $K$ affects the underlying bound. To this end, we fix the number of samples $s=10000$ and vary $K\in\{100, 500, 1000, 5000, 10000\}$. The numerical results, averaged over $50$ different trials with learning rate $\gamma=0.01$, are reported in Tab.~\ref{tab:results_varyingK}, where $\Delta x_K\eqdef\|x^{K+1}-x^\star\|$ and some reference $x^\star$ is computed by using Gurobi \cite{gurobi}. As expected, the dependence on $K$ in the stability bound from Lemma~\ref{lemma:FB_unif_stab_general} results in rather loose certificates when $K$ is large. Therefore, when the operator $\mc B$ is merely cocoercive, it seems preferable to run Algorithm~\ref{alg:FBmod} for only a few iterations, e.g., $K\in[500,1000]$. Based on our numerical experience, this approach yields a solution $x^{K+1}$ that is empirically close to some $x^\star$, and it is accompanied by a tighter probabilistic certificate. This observation is consistent to what already pointed out in \cite{hardt2016train}.

\section{Conclusion and outlook}\label{sec:conclusion}
By focusing on the problem of finding a zero of the sum of two operators, where one is either unavailable in closed form or computationally expensive to evaluate, we have derived rigorous certificates on the quality of solutions produced by data-based \gls{FB} operator splitting schemes. As frequently happens in stochastic optimization, we have indeed approximated such an expensive operator by means of a finite number of samples. Since exact convergence to a zero should not be expected in this limited information setting, we have established probabilistic bounds on the distance between a true zero and the \gls{FB} output, without making specific assumptions about the underlying data distribution. This has been made possible through a careful design of a tailored loss function representative for our purposes. We have then proved uniform stability of the data-driven \gls{FB} \gls{wrt} such a loss function under different monotonicity assumptions on the operators involved. Once derived computable expressions for an $\varepsilon$-zero that hold true with high confidence, we have finally applied our results to a popular \gls{SNE} seeking algorithm based on the \gls{FB} scheme.

Future work will analyze the stability properties for other operator splitting methods, e.g., Douglas-Rachford \cite[Ch.~25]{bauschke2017correction}, as well as possible randomized variants \cite{hardt2016train} of the resulting data-driven schemes. Aiming at improving our probabilistic bounds, one can also investigate the problem considered in this paper under different lenses, leveraging tools from \gls{PAC}-value iteration or oracle-based complexity measures. Finally, along the line of \cite{mohri2010stability,zhang2019mcdiarmid}, possible relaxations of the \gls{iid} requirement on the dataset at hand could also be explored.

\bibliographystyle{IEEEtran}
\bibliography{FBstability}


\end{document}